\UseRawInputEncoding
\documentclass[11pt]{article}

\usepackage[T1]{fontenc}
\usepackage{amsmath, amscd, amsfonts, amssymb, amsthm, mathtools}
\usepackage{graphicx, tikz}
\usepackage{dsfont, bbm}
\usepackage{enumerate}
\usepackage[title]{appendix}
\usepackage{hyperref}
\usepackage{fullpage}
\usepackage{dsfont}
\usepackage[english]{babel}
\usepackage[nottoc]{tocbibind}

\theoremstyle{plain}
\newtheorem{theorem}{Theorem}[section]
\newtheorem{lemma}[theorem]{Lemma}
\newtheorem{corollary}[theorem]{Corollary}
\newtheorem{proposition}[theorem]{Proposition}

\newtheorem*{theorem*}{Theorem}
\newtheorem*{proposition*}{Proposition}

\theoremstyle{definition}
\newtheorem{definition}[theorem]{Definition}
\newtheorem{assumption}[theorem]{Assumption}

\theoremstyle{remark}
\newtheorem{remark}[theorem]{Remark}
\newtheorem*{remark*}{Remark}

\newcommand{\prob}{\mathbb{P}}
\newcommand{\N}{\mathbb{N}}
\newcommand{\R}{\mathbb{R}}

\newcommand{\E}{\mathbb{E}}
\newcommand{\eps}{\epsilon}

\newcommand{\indic}{\mathds{1}}

\date{ }

\begin{document}

\begin{center}
\Large\bf
Phase transition for the bottom singular vector of rectangular random  matrices
\end{center}

\vspace{0.5cm}
\renewcommand{\thefootnote}{\fnsymbol{footnote}}
\hspace{5ex}	
\begin{center}
 \begin{minipage}[t]{0.4\textwidth}
\begin{center}
Zhigang Bao\footnotemark[1]  \\
\footnotesize {University of Hong Kong}\\
{\it zgbao@hku.hk}
\end{center}
\end{minipage}
\begin{minipage}[t]{0.4\textwidth}
\begin{center}
Jaehun Lee\footnotemark[2]  \\ 
\footnotesize {City University of Hong Kong}\\
{\it jaehun.lee@cityu.edu.hk}
\end{center} 
\end{minipage}
\end{center}
\vspace{2ex}
\begin{center}
\begin{minipage}[t]{0.5\textwidth}
\begin{center}
Xiaocong Xu\footnotemark[3]  \\
\footnotesize {Hong Kong University of Science and Technology}\\
{\it xxuay@connect.ust.hk}
\end{center}
\end{minipage}
\end{center}

\footnotetext[1]{Supported by Hong Kong RGC Grant 16303922, NSFC12222121 and NSFC12271475}
\footnotetext[2]{Supported by  NSFC No. 2023YFA1010400}
\footnotetext[3]{Supported by  Hong Kong RGC Grant 16303922 and NSFC12222121}

\vspace{2ex}
\begin{center}
 \begin{minipage}{0.8\textwidth} {
 In this paper, we consider the rectangular random matrix $X=(x_{ij})\in \mathbb{R}^{N\times n}$ whose entries are iid with tail $\mathbb{P}(|x_{ij}|>t)\sim t^{-\alpha}$ for some $\alpha>0$. We consider the regime $N(n)/n\to \mathsf{a}>1$ as $n$ tends to infinity. Our main interest lies in the right singular vector corresponding to the smallest singular value, which we will refer to as the "bottom singular vector", denoted by $\mathfrak{u}$. 
In this paper, we prove the following  phase transition regarding the localization length of $\mathfrak{u}$:  when $\alpha<2$ the localization length is $O(n/\log n)$;  when $\alpha>2$ the localization length is of order $n$. Similar results hold for all right singular vectors around the smallest singular value. The variational definition of the bottom singular vector suggests that the mechanism for this localization-delocalization transition when $\alpha$ goes across $2$ is intrinsically different from the one for the top singular vector when $\alpha$ goes across $4$.
}
\end{minipage}
\end{center}

\vspace{2ex}
\section{Introduction}

\noindent
Given a large random matrix, one may wonder how the mass of its eigenvectors/singular vectors is distributed over the coordinates. A widely studied pair of properties regarding the profile of random matrix eigenvectors/singular vectors is localization/delocalization. Roughly speaking, a vector is localized if its mass is mainly distributed among a small portion of the coordinates, while it is delocalized if the mass spreads out among a large portion of the coordinates, resulting in a flatter profile. For some classical mean field models like GUE and GOE, it is well known that all of their eigenvectors are uniformly distributed on the unit sphere and are thus completely delocalized. However, for various other random matrix models, such as the random Schr\"{o}dinger operator, random band matrices, sparse matrices, and heavy-tailed random matrices, localized eigenvectors may exist. Indeed, the problem of localization-delocalization transition, induced by the  changes in certain key parameters, is of great interest in Random Matrix Theory. We refer to Section \ref{s.ref review} for a more detailed review of the related literature.

Despite the rough descriptions of localization and delocalization as above, the mathematical definitions of them are not unique. For delocalization, two major definitions have been adopted in the literature. They are  {\it sup-norm delocalization} and {\it no-gaps delocalization}. For a unit vector $u=(u_1, \ldots, u_n)\in \mathbb{C}^n$, we say that $u$ satisfies the sup-norm delocalization if $\|u\|_\infty\leq \log^\gamma n/n$ for some constant $\gamma>0$.
On the other hand, no-gaps delocalization is defined as the existence of a nice function $f:(0,1) \to (0,1)$ such that for any $\epsilon \in (0,1)$ and any subset $I \subseteq \{1,2,\ldots,n\}$ with $|I| \ge \epsilon n$, one has
$(\sum_{i \in I} |u_i|^{2})^{1/2} \ge f(\epsilon)$.
For instance, for the eigenvectors of random matrices such as the Wigner matrices, the sup-norm delocalization was first established in \cite{ESY09-1}, and the no-gaps delocalization was first established in \cite{RV16}. 
One may easily notice that these two types of delocalization results cannot imply each other, and indeed the proof mechanisms are rather different.
For localization, there have been also various  definitions, but most of them are formulated as the decay of the Green function or the fractional moment of Green function. In case of the random band matrix, we refer to \cite{Sch09} for instance.

In this paper, we aim at establishing a localization-delocalization transition for the right singular vector associated with the smallest singular value (bottom singular vector) of a tall random matrix  when the fatness of the matrix entry distribution varies. Despite a localization-delocalization transition for the top singular vector, i.e., the right singular vector associated with the largest singular value,  induced by the fatness of the entry distribution, being well-known, the mechanism for the phase transition for the bottom one seems rather different. In the sequel, we state our model and objectives more precisely.

Let $X = (x_{ij})$ be an $N \times n$ rectangular random matrix with i.i.d.~entries, where $N = \lceil \mathsf{a} \cdot n \rceil$ for a constant $\mathsf{a} > 1$. Assume that each $x_{ij}$ satisfies
$
	\prob\{|x_{ij}| > t\} \sim t^{-\alpha},  \alpha > 0.
$
It is known that for the largest singular values of $X$, a phase transition occurs when $\alpha$ goes across $4$. When $\alpha>4$,  the largest singular values of $N^{-\frac12}X$ converge to the square root of the right end point of the Marchenko-Pastur law (MP law), $1+\sqrt{\mathsf{a}}$, and the fluctuation is given by an Airy point process and follows the Tracy-Widom type law \cite{BS88, DY1}. When $\alpha<4$, it is known that the largest singular values of $X$ are asymptotically given by the largest entries of $X$, and are thus Poisson and fluctuate on a scale of $N^{2/\alpha}$; see \cite{So06, ABP09}. From the discussions in \cite{So06, ABP09} and the definition $s_{\max}(X)=\sup_{\|u\|_2=1 }\|Xu\|_2$, one can readily see that the top singular vectors should be completely localized in the sense that most of their mass are distributed on the locations corresponding to a few largest entries of the $X$. 
When one turns to the smallest singular values, the phase transition is drastically different from the largest ones in the sense that it is more robust and also more sophisticated. From \cite{Tikhomirov15}, one learns that a second moment condition $\mathbb{E}|x_{ij}|^2=1$ is enough for the convergence of the smallest singular value to the square root of the left edge of the  MP law, $1-\sqrt{\mathsf{a}}$, and thus $\alpha>2$ is sufficient for this convergence. It has also been shown recently in \cite{BLX23+} that a phase transition for the fluctuation of the smallest singular value occurs when $\alpha$ goes across $8/3$. More specifically, when $\alpha> 8/3$, the fluctuation is given by Tracy-Widom law, but when $\alpha\in (2, 8/3)$, the fluctuation is Gaussian.  In contrast, the limiting behavior of the smallest singular value is much less known in the case $\alpha\in (0,2)$.  To the best of our knowledge, in the regime $\alpha\in (0,2)$, the following lower and upper bounds can be obtained  from the literature: as a special case in \cite{Tikhomirov16}, one has a lower bound 
$s_{\min}(X)\gtrsim \sqrt{n}$ with high probability, while as a consequence of the global law in \cite{BDG09}, one can derive an upper bound $s_{\min}(X)\ll n^{1/\alpha}$. The mechanism to identify the typical size and precise limiting behavior of $s_{\min}(X)$ seems still missing.

Recall the variational formula for the smallest singular value $s_{\min}(X)$ of $X$,
\begin{equation}\label{eq: 97}
	s_{\min}(X) = \inf_{\lVert u \rVert_{2}=1} \lVert X u \rVert_{2},
\end{equation}
In the case $\alpha>2$, except for a few leading columns, the $\ell^2$-norms of all remianing columns are of order $\sqrt{n}$. Heuristically, the bottom singular vector $\mathfrak{u}$, as a minimizer of the variational problem in (\ref{eq: 97}), serves as the coefficient vector for the linear combination of those columns with similar lengths. Hence, a delocalized $\mathfrak{u}$ that can induce nontrivial cancellations among these columns is expected. In the case $\alpha<2$, the picture is drastically different. It is easy to check that the majority of the columns have lengths of order $n^{1/\alpha}$, and the leading ones are of order $n^{2/\alpha}$. More importantly, there are also $o(n)$ columns that have lengths much smaller than $n^{1/\alpha}$. For instance, in the case that $x_{ij}$'s are exactly $\alpha$-stable, according to the exact density of the one-sided L\'{e}vy stable distribution near the point $0$ in \cite{PG10}, one may see that there are $o(n)$ columns of $X$, whose $\ell^2$-norms are all of order $n^{1/\alpha} (\log n)^{(\alpha-2)/2\alpha}$. A natural conjecture is then that $\mathfrak{u}$ may tend to concentrate on the coordinates of these small columns, as a minimizer in (\ref{eq: 97}).
 If we choose a vector $u$ which concentrates on the coordinates of the $o(n)$ amount of  small columns, one can get an upper bound of $s_{\min}(X)$ that is much smaller than $n^{1/\alpha}$. Certainly,  it does not imply that $\mathfrak{u}$ has to be mainly supported on these coordinates, since the possible collinearity of typical and large columns may induce the smallness of the smallest singular value as well. Nevertheless, it motivates us to consider the possibility of the localization of $\mathfrak{u}$ when $\alpha\in (0,2)$. Indeed, in this paper, we will take advantage of the small columns to obtain a localization result for $\mathfrak{u}$ in case $\alpha\in (0,2)$ and further show a localization-delocalization transition of $\mathfrak{u}$ when $\alpha$ goes across $2$. 
 
 We remark here that the definitions of localization and delocalization adopted in our results are different from previous ones, but they are inspired by the definition of no-gaps delocalization from \cite{RV16}. More specifically, the localization and delocalization properties in this paper are essentially compressibility and incompressibility of vectors in \cite{RV08}.  Roughly speaking, a vector is compressible if it is sufficiently close to a sparse vector, and it is incompressible if it is not compressible.  We further remark that in the delocalization case,   the relation between the eigenvector/singular vector profile with the lower bound of the smallest singular value of certain random matrices has been readily seen in \cite{RV16}. And it is also well-known that the concepts of compressible and incompressible vectors are widely used in the works of the lower bound of the smallest singular values of random matrices; see \cite{LPRTj05, RV09, Tikhomirov20, Vershynin14, BR17, Livshyts21, Cook18, LTV21, GT23} for instance. Indeed, our discussions on both localization and delocalization regimes are inspired by the works related to the smallest singular values, especially \cite{Tikhomirov16} and \cite{RV16}. 
 
 Our main results will be detailed in the next subsection, and  an outline of the proof strategy will be stated in Section \ref{ss.proof strategy}.

\subsection{Results}

\begin{definition}[Rectangular random matrix]\label{def: 78}
Let $X=(x_{ij})$ be an $N\times n$ random matrix with i.i.d.~entries.
Assume $N = \lceil \mathsf{a}\cdot n \rceil$ for a constant $\mathsf{a}>1$.
Suppose that  for a constant $\alpha>0$, we have
\begin{equation}\label{eq: 81}
	C_{\ell}t^{-\alpha}\le \prob\{ |x_{11}| > t \} \le C_{u}t^{-\alpha}, \quad t\ge T_{0},
\end{equation}
with some constants $C_{\ell},C_{u},T_{0}>0$.
\end{definition}

\begin{assumption}\label{assump: 89}
The random variable $x_{11}$ is symmetrically distributed.
\end{assumption}

\begin{remark}
The assumption of symmetry in the distribution is made to streamline the proof.
See \eqref{eq: 248} and \eqref{eq: 1125}, along with their descriptions, for more details.
\end{remark}

Let us write $[n] \equiv \{1, 2, \ldots, n\}$ for brevity. For a vector $\mathbf{x} \in \mathbb{R}^n$ and a subset $I \subseteq [n]$, we denote by $\mathbf{x}_I\in \mathbb{R}^{|I|}$ the vector obtained from $\mathbf{x}$ by deleting the components with indices in $I^c$. Our first result shows that the mass of the bottom singular vector is concentrated on the components of order larger than or equal to $(\log n/n)^{1/2}$, when $\alpha\in (0,2)$.

\begin{theorem}\label{thm: ssv}
Let $X = (x_{ij})$ be as in Definition \ref{def: 78} with $\alpha \in (0, 2)$. Suppose Assumption \ref{assump: 89} holds. Let $\mathfrak{u}$ be the bottom singular vector of $X$. For any $\delta \in (0, 1)$ and $K > 0$, there exists a constant $c_{\ref{thm: ssv}}\equiv c_{\ref{thm: ssv}}(\alpha, \delta, K) > 0$ such that the following holds. Define the subset $\hat{I} \subseteq [n]$ depending on $\mathfrak{u}$ by
\begin{equation*}
    \hat{I} \equiv \hat{I}(\mathfrak{u}) = \Big\{ i \in [n] : |\mathfrak{u}_i| > \sqrt{\frac{c_{\ref{thm: ssv}} \log{n}}{n}} \Big\}.
\end{equation*}
We have $\lVert \mathfrak{u}_{\hat{I}} \rVert_2 \ge \sqrt{1 - \delta} $ with probability at least $1 - O(n^{-K})$. 
\end{theorem}

\begin{remark}
In fact, we can slightly extend Theorem \ref{thm: ssv} to the case where the $x_{ij}$'s are independent but not necessarily identically distributed. Instead of \eqref{eq: 81}, we may assume that the following bound on the tail distribution holds uniformly in $i$ and $j$:
$
    C_{\ell} t^{-\alpha} \le \prob\{ |x_{ij}| > t \} \le C_{u} t^{-\alpha}, t \ge T_{0}.
$
Then, the same conclusion can be obtained without any major modifications to the proof.
Additionally, we may also consider regularly varying tails: $\prob\{ |x_{ij}| > t \} = L(t) t^{-\alpha}$, where $L$ is a slowly varying.
\end{remark}

\begin{remark}
The mass of $\mathfrak{u}$ is mostly concentrated on the set $\hat{I}$ if we choose $\delta$ to be sufficiently small, and by the Markov inequality, the cardinality of $\hat{I}$ is bounded by $O(n / \log{n})$. From the simulation study in Figure \ref{fig:1}, it seems that unlike the top singular vector in case $\alpha\in (0,4)$, which is completely localized, the localization length of $\mathfrak{u}$ seems much longer. Nevertheless, it is not clear from our proof if the current upper bound $n / \log{n}$ is close to be optimal.  We would like to mention here,  it was observed in \cite{BVZ23} that the eigenvectors of a non-Hermitian Toeplitz matrix with a small random perturbation are localized at the scale of $n / \log{n}$ as well. 
\end{remark}

~

\begin{figure}[t]
	\begin{minipage}[t]{0.3\textwidth}
		\includegraphics[width=\textwidth]{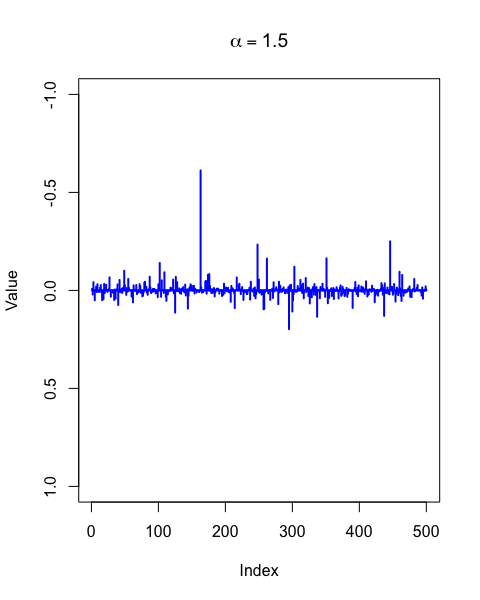}
	\end{minipage}
	\begin{minipage}[t]{0.3\textwidth}
		\includegraphics[width=\textwidth]{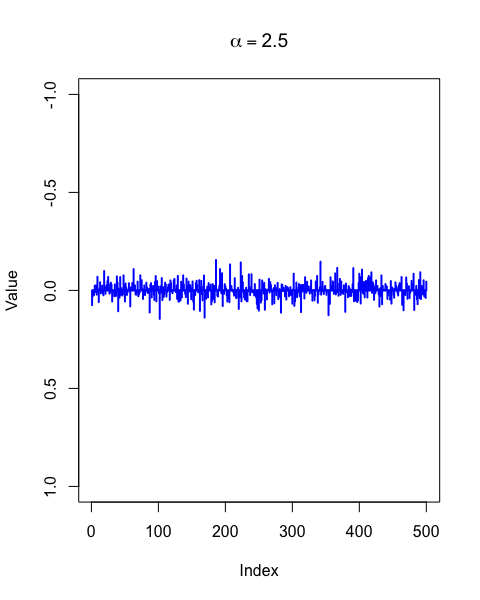}
	\end{minipage}
	\begin{minipage}[t]{0.3\textwidth}
		\includegraphics[width=\textwidth]{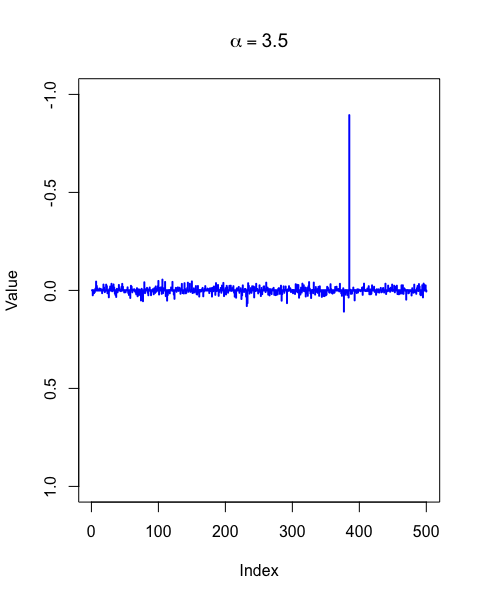}
	\end{minipage}
	\caption{Profiles of bottom singular vectors ($\alpha = 1.5, 2.5$) and top singular vector ($\alpha = 3.5$)}
	\label{fig:1}
\end{figure}

The above theorem implies that there is a small portion of coordinates which carry a majority of the total mass in $\mathfrak{u}$. In particular, it implies the failure of no-gaps delocalization when $\alpha \in (0, 2)$. We also note that this result does not directly exclude the possibility of sup-norm delocalization.

Let $\hat{I}$ be as in Theorem \ref{thm: ssv}. The following corollary is easily verified by setting $\hat{J} = [n] \backslash \hat{I}$.

\begin{corollary}\label{cor: 179}
Suppose the assumptions in Theorem \ref{thm: ssv} hold. For any arbitrarily small $\delta > 0$, we can find a subset $\hat{J} \subset [n]$ such that $|\hat{J}| = n - o(n)$ but $\lVert \mathfrak{u}_{\hat{J}} \rVert_2 < \sqrt{\delta} $ with high probability.
\end{corollary}

Moreover, we may extend the above results to  the right singular vector corresponding to the $k$-th smallest singular value with $k \ll n$.

\begin{theorem}\label{thm: 220}
Let $X = (x_{ij})$ be as in Definition \ref{def: 78} with $\alpha \in (0, 2)$. Suppose Assumption \ref{assump: 89} holds. Let $\mathfrak{b} \in (0, 1/2)$ be a constant. Consider $k=O(n^{1-2\mathfrak{b}})$. Let $\mathfrak{u}^{(k)}$ be the unit right singular vector corresponding to the $k$-th smallest singular value. For any $\delta \in (0, 1)$ and $K > 0$, there exists a constant $c_{\ref{thm: 220}} > 0$ such that the following holds. Define the subset $\hat{I}_{k} \subset [n]$ depending on $\mathfrak{u}^{(k)}$ by
\begin{equation*}
    \hat{I}_{k} \equiv \hat{I}_{k}(\mathfrak{u}^{(k)}) = \Big\{ i \in [n] : |\mathfrak{u}^{(k)}_{i}| > \sqrt{\frac{c_{\ref{thm: 220}} \log{n}}{n}} \Big\}.
\end{equation*}
We have $\lVert \mathfrak{u}^{(k)}_{\hat{I}_{k}} \rVert_2 \ge \sqrt{1 - \delta} $ with probability at least $1 - O(n^{-K})$. The constant $c_{\ref{thm: 220}}$ may depend on $\alpha$ and $\delta$. \end{theorem}

\begin{corollary}\label{cor: 233}
Suppose the assumptions in Theorem \ref{thm: 220} hold. Let $\mathfrak{b} \in (0, 1/2)$ be a constant. Consider $k=O(n^{1-2\mathfrak{b}})$.  For any arbitrarily small $\delta > 0$, we can find a subset $\hat{J}_{k} \subset [n]$ such that $|\hat{J}_{k}| = n - o(n)$ but $\lVert \mathfrak{u}^{(k)}_{\hat{J}_{k}} \rVert_2 < \sqrt{\delta}$ with high probability.
\end{corollary}

~

Next, let us consider the case $\alpha > 2$. We find that the conclusion of Corollary \ref{cor: 179} does not hold, which creates a striking contrast with the case $\alpha \in (0, 2)$. 

\begin{theorem}\label{thm: 133}
Let $X = (x_{ij})$ be as in Definition \ref{def: 78} with $\alpha > 2$ and assume that $\mathbb{E} x_{11}^{2} = 1$. Suppose Assumption \ref{assump: 89} holds. Recall that $\mathfrak{u}$ is the bottom singular vector of $X$. For any $K > 0$, there exist constants $\eps_{\mathsf{a}}, \delta_{\mathsf{a}} \in (0, 1/2)$ such that the following holds with probability at least $1 - O(n^{-K})$. For $\eps \in (0, \eps_{\mathsf{a}})$, we have $\lVert \mathfrak{u}_{I} \rVert_2 \ge \delta_{\mathsf{a}} $ for all $I \subset \{1, 2, \ldots, n\}$ with $|I| = \lceil (1 - \eps) n \rceil$.
\end{theorem}

\begin{corollary}\label{cor: 188}
Suppose the assumptions in Theorem \ref{thm: 133} hold.
Consider $k = o(n)$.
Recall that $\mathfrak{u}^{(k)}$ is the right singular vector corresponding to the $k$-th smallest singular value.
For any $K > 0$, let $\eps_{\mathsf{a}}$ and $\delta_{\mathsf{a}}$ be as in Theorem \ref{thm: 133}. Then, the following holds with probability at least $1 - O(n^{-K})$.
For $\eps \in (0, \eps_{\mathsf{a}})$, we have $\lVert \mathfrak{u}^{(k)}_{I} \rVert_2 \ge \delta_{\mathsf{a}}$ for all $I \subset \{1, 2, \ldots, n\}$ with $|I| = \lceil (1 - \eps) n \rceil$.
\end{corollary}

\begin{remark}
This result is weaker than no-gaps delocalization in an apparent sense, since it only tells us that any large set of coordinates (those indexed by $I$ with $|I|=\lceil (1 - \epsilon) n \rceil$) will have a total mass no less than $\delta_{\mathsf{a}}$, but does not provide a lower bound for $\|\mathfrak{u}_I\|_2$ for any $|I|\geq \epsilon n$. However, this result does serve as the opposite case of the conclusion in Corollary \ref{cor: 179}, since it states that any small set of coordinates indexed by $I^c$ cannot hold the majority of the mass greater than $1-\delta_{\mathsf{a}}$. We would like to remark here that it is not yet known whether no-gaps delocalization holds for $\mathfrak{u}$ in the case of rectangular random matrices, even for those with lighter tails such as sub-Gaussian distributions, despite the apparent result in the Gaussian case. 
\end{remark}

\subsection{Related works} \label{s.ref review}

The localization-delocalization transition is also well-known as the Anderson transition in physics, which is used to model the metal-insulator transitions. Establishing such a phase transition is a fundamental task to various models such as random Schr\"{o}dinger operators, random band matrices, sparse random matrices, and heavy-tailed random matrices.

For  dense mean field models like Wigner matrices with sufficiently light tailed entries, it had been widely believed that their eigenvectors should behave similarly to GUE/GOE and thus are universal. Consequently, one would believe that all eigenvectors of such Wigner matrices should satisfy both sup-norm delocalization and no-gaps delocalization, and further they are asymptotically Gaussian. 
Indeed, in the last decade, a major achievement in random matrix theory, is the proof of the universality of these results for all such Wigner matrices and similar models. The first sup-norm delocalization result was established in \cite{ESY09-1} after an initial weaker result in \cite{ESY09-2}, since then various similar results have been obtained for related models to various extent \cite{EYY12, EYY12b, TV11, AEK17, VW15, BL22}.  
The first no-gaps delocalization result was established in \cite{RV16} for a rather general class of square matrices, including Wigner matrices, iid non-Hermitian matrices as typical examples, under the assumption that the operator norms of the matrices are bounded by $C\sqrt{n}$. In case of Wigner matrices, it would require the existence of the 4-th moment of the entries. We also refer to \cite{LT20, LOr20, OrVW16, DLL11, ERS17, NTV17, LL21, AB11} for related study. Finally, we  refer to \cite{BY17, KY13, TV12, MY22, CES22, benigni2020eigenvectors} for more precise distribution of the eigenvector statistics for Wigner matrices and related models.

If one turns to consider random matrix models which are non mean field, or sparse, or heavy-tailed, it is widely believed that a localization-delocalization transition will occur for the eigenvectors in certain regime of the spectrum. Such a phase transition for eigenvectors often comes along with a phase transition of eigenvalues statistics from the random matrix statistics to Poisson statistics. Without trying to be comprehensive, we refer to \cite{And78, AW15, CL90, FS83, AM93, Spe88, DS20, CKM87} for random Schr\"{o}dinger operator, \cite{FM91,MFDQS96, Sch09, BE17, YYY21+, YYY22, BYY20, EK11, EKYY13-2, CPSS24, HM19, CS22+} for random band matrices, and \cite{EKYY13, ADK21,ADK23,ADK24, HKM19} for the sparse random matrices, regarding the progress of the localization-delocalization  transition.  In the rest of this subsection, we will then mainly focus on the literature of the heavy-tailed random matrices which are closely related to our study. 

Depending on the spectral regime of interest, one often regards a Wigner matrix as heavy-tailed if $\alpha\in (0,2)$ (bulk regime) or $\alpha\in (0,4)$  (edge regime). It has now been well understood what happens towards the top eigenvalues and eigenvectors of Wigner matrices when $\alpha$ goes across $4$, as we mentioned earlier. Recently, there has also been some major progress in the understanding of the bulk regime when $\alpha\in (0,2)$.  A Wigner matrix whose entries fall within the domain of attraction of an $\alpha$-stable law is known as a L\'{e}vy matrix. The global spectral distribution obtained in \cite{BG08, BCC11} is no longer semicircle law. Instead, it is a symmetric heavy-tailed law with an unbounded support. Many studies have be done on understanding the eigenvalues around a fixed energy level and the corresponding eigenvectors. Numerical simulations \cite{CB94} suggested that the local spectral statistics of L\'{e}vy matrices exhibit a phase transition from those of GOE at low energies to those of a Poisson process at high energies when $\alpha < 1$. 
This prediction from \cite{CB94} was later revised in \cite{TBT16} using the supersymmetric method, resulting in the following findings:
(i) For $1 \leq \alpha < 2$, at any energy level $E \in \mathbb{R}$, the local spectral statistics near $E$ converge to those of the GOE, and the corresponding eigenvectors are completely delocalized.
(ii) For $0 < \alpha < 1$, there exists a mobility edge $E_{\alpha}$ such that for energy levels $E < E_{\alpha}$, the local spectral statistics near $E$ resemble those of the GOE and the eigenvectors are delocalized, whereas for $E > E_{\alpha}$, the local spectral statistics near $E$ converge to a Poisson point process and the eigenvectors are localized. The GOE local statistics and complete eigenvector delocalization for $\alpha \in (1, 2)$ were rigorously proven in \cite{ALY}.
For $\alpha \in (0, 1)$, the authors of \cite{BG13, BG17} demonstrated that eigenvectors with sufficiently high energy levels are localized, while those with sufficiently low energy levels are partially delocalized. 
More recently, \cite{ABL} addressed a sharp phase transition from a GOE-delocalized phase to a Poisson-localized phase when $\alpha \in (0, 1)$. We emphasize that these results are for certain fixed energy levels, and thus do not concern with the extreme eigenvalues. 

Due to the symmetric nature of the spectrum of the Wigner matrices, one does not need to distinguish the study of the largest eigenvalue and its eigenvector from the smallest one. Nevertheless, when one turns to the singular values of the rectangular matrix, also equivalent to the square roots of the spectrum of the sample covariance matrices, the top eigenvalue/eigenvector and bottom eigenvalue/eigenvector may have rather asymmetric behavior in response to the change of $\alpha$. While it is known that the top eigenvalue/eigenvector behave similarly to those of Wigner matrices \cite{DY1,BKY16}, the bottom one is much less known.  Regarding the bottom eigenvalue, the limiting behavior has been well understood from \cite{Tikhomirov15, BLX23+} when $\alpha>2$. 
When $\alpha\in (0,2)$, the global law of the spectrum was obtained in \cite{BDG09}. Specifically, the spectrum of $n^{-2/\alpha}XX^*$ converges weakly to a limiting measure whose density does not vanish in any neighborhood of $0$.  This result implies $s_{\min}(X)=o(n^{1/\alpha})$ with high probability. This provides a rough upper bound of $s_{\min}(X)$. Regarding the lower bound, we can learn from \cite{Tikhomirov16} that $s_{\min}(X)\gtrsim \sqrt{n}$ with high probability, under our assumptions in Definition \ref{def: 78}. We remark here that Tikhomirov's result in \cite{Tikhomirov16} holds under much more general setting where no moment condition of $x_{ij}$ is imposed and instead only a mild anti-concentration assumption is needed. However, there is no evidence that the lower bound is tight for our more specific model. It is an appealing problem to identify the true order and  the precise limiting behavior of $s_{\min}(X)$.

\subsection{Notation}

 We use $\|u\|_2$ and $\|u\|_\infty$ to denote the $\ell^2$ and $\ell^\infty$-norms of a vector $u$, respectively. We further use   $\|A\|$ for the operator norm of a matrix $A$.  We use $C$ to denote some generic (large) positive constant.   For any positive integer $k$, let $[k]$ be the set $\{1,\ldots,k\}$. We also denote by $u|_{I}$ or $u_I$ the restriction of a vector $u$ on the index set $I$, i.e., the vector obtained by deleting the components indexed by $I^c$ from $u$.  Similarly, let $A_{L}$ denote the minor of a matrix $A$ consisting of the columns indexed by $L$ for any subset $L \subset [n]$. 

For any random vector $\mathbf{x}\in\mathbb{R}^{k}$, concentration function of $\mathbf{x}$ is defined as 
\begin{equation}\label{eq: 345}
	\mathcal{Q}(\mathbf{x},t) = \sup_{\nu\in\mathbb{R}^{k}}\prob\{\lVert \mathbf{x}-\nu \rVert_2\le t\},
	\quad t> 0.
\end{equation}
 For a subset $E\subset\mathbb{R}^{n}$ and $Y\in \mathbb{R}^{N\times n}$, we write $Y(E)=\{Yw\in\mathbb{R}^{N} : w\in E\}$.
 For a vector $y\in\mathbb{R}^{N}$ and a subset $E'\subset\mathbb{R}^{N}$, we denote by $\mathrm{d}(y,E')$ the Euclidean distance between $y$ and $E'$, i.e., $\mathrm{d}(y,E')=\inf\{\lVert y - w' \rVert_{2} : w'\in E'\}$. Similarly, for subsets $E, E' \subset \mathbb{R}^{n}$, we define $\mathrm{d}(E, E') = \inf\{\lVert w - w' \rVert_{2} : w \in E, w' \in E'\}$.

\section{Proof Strategy} \label{ss.proof strategy}

We will now outline the proof strategies for two main results: Theorem \ref{thm: ssv} and Theorem \ref{thm: 133}.

~

For  Theorem \ref{thm: ssv}, we will prove it by contradiction. 
Recall $s_{\min}(X) = \lVert X \mathfrak{u} \rVert_{2}$. We aim to show that, with high probability, most of the mass of $\lVert \mathfrak{u} \rVert_{2}$ is concentrated on coordinates larger than $(c \log{n} / n)^{1/2}$ for some small constant $c > 0$. Recall the definition of $\hat{I}$ from Theorem \ref{thm: ssv}. Consider the event $\lVert \mathfrak{u}_{\hat{I}} \rVert_{2}^{2} < 1 - \delta$, which implies $\lVert \mathfrak{u}_{\hat{I}^{c}} \rVert_{2}^{2} > \delta$. On this event, using the variational formula \eqref{eq: 97} and Proposition \ref{prop: 240} (which we will prove later), we will prove that, with high probability,
\begin{equation}\label{eq: 257}
    s_{\min}(X) > M(c) n^{\frac{1}{\alpha}} (\log{n})^{\frac{\alpha-2}{2\alpha}},
\end{equation}
where $M(c) > 0$ is a constant depending on $c$, with $M(c) \uparrow \infty$ as $c \downarrow 0$.

However, on the other hand, we will show in Theorem \ref{thm: upper bound} the following upper bound for \( s_{\min}(X) \):
\begin{equation}\label{eq: 262}
    s_{\min}(X) < C n^{\frac{1}{\alpha}} (\log{n})^{\frac{\alpha-2}{2\alpha}}, 
\end{equation}
for some constant $C>0$. By choosing \( c > 0 \) sufficiently small in \eqref{eq: 257}, we see that \eqref{eq: 257} contradicts \eqref{eq: 262}.

The proof of Proposition \ref{prop: 240} closely follows the method introduced in \cite{Tikhomirov16}. In \cite{Tikhomirov16}, Tikhomirov develops an approach for the lower bound of the smallest singular value of a rectangular matrix without any moment condition, based on the approach developed in \cite{LPRTj05,RV08,RV09} where sufficiently light-tailed assumptions are imposed and also the estimates of the distance between a random vector and a fixed linear subspace in \cite{RV15} where no moment assumption is imposed as well. Despite estimating $s_{\min}(X)$ requires one to find the minimizer of $\|Xu\|_2$ among all $u\in \mathbb{S}^{n-1}$, the method in \cite{Tikhomirov16} and the mentioned previous reference can be used to estimate $\inf_{u\in \mathcal{S}_0} \|Xu\|_2$ over a much more general subset $\mathcal{S}_0\subseteq \mathbb{S}^{n-1}$. Especially, all the proofs in the above reference involve a certain decomposition of $\mathbb{S}^{n-1}$ into subsets, such as subsets of compressible and incompressible vectors.  In order to establish (\ref{eq: 257}), we shall apply the approach in \cite{Tikhomirov16} to some $\mathcal{S}_0$ which contains $u$'s that are not sufficiently localized, i.e., $\lVert u_{\hat{I}^{c}} \rVert_{2}^{2} > \delta$. Alternatively, they are incompressible.  Moreover, in order to get an improved lower bound in (\ref{eq: 257}), rather than the $\sqrt{N}$ bound in \cite{Tikhomirov16}, we shall actually work with a rescaled matrix $\widetilde{X} = N^{-\frac{1}{\alpha}+\eps_{N}}X$ which essentially resembles a sparse matrix with sparsity around the critical scale $O(\log n/n)$, after getting rid of heavy-tailed part, in the spirit of \cite{Tikhomirov16}.
To the best of our knowledge, this level of sparsity has not been addressed in the context of the smallest singular value of a sparse rectangular random matrix; see \cite{GT23}. In this article, we partially address this issue by considering a variational problem \eqref{eq: 97} restricted to a subset of incompressible vectors, i.e., $\mathcal{S}_{0}$, which will be defined later in \eqref{eq: 207}. 

 In order to get a quantitative upper bound of $s_{\min}(X)$ in Theorem \ref{thm: upper bound}, we first observe from the basic Cauchy interlacing that $s_{\min}(X)$ is bounded by the smallest (non-zero) singular value of any column minor of $X$. Here by ``column minor" we mean a  submatrix 
obtained by keeping certain columns of $X$ only and deleting the others. We will then take the advantage of the small columns,  i.e., the columns with small entries/small $\ell^2$-norms. We will then simply use the operator norm of this minor with small columns to bound  $s_{\min}(X)$ from above. The operator norm of this minor can be estimated by using Seginer's result (Proposition \ref{prop: Seginer}).

~

For the proof of Theorem \ref{thm: 133},  
we draw inspiration from the approach outlined in \cite{RV16} and adapt it to  our problem. The approach developed in \cite{RV16} works well for the eigenvector of a square matrices where the entries are fully independent or independent up to symmetry. It is based on the eigenvalue equation 
$(A-\lambda)v=0$ for an $n\times n$ random matrix $A$ with eigenpair $(\lambda, v)$. Writing $B=A-\lambda$, followed by a decomposition $B_I v_I+B_{I^c} v_{I^c}=0$, one can prove that $\|v_I\|_2$ cannot be too small if $\|B_I\|\leq C\sqrt{n}$ and $s_{\min} (B_{I^c})\geq c\sqrt{n}$ for some positive constants $C$ and $c$. Here $B_L$ is the minor of $B$ consisting of the columns indexed by $L\subset [n]$ for $L=I$ or $I^c$.  Using the argument for all $I$ with cardinality $\varepsilon n$ and estimate the probability for the union of the bad events over all such index set $I$, one can prove the no-gaps delocalization.

When one applies this idea to prove our Theorem \ref{thm: 133}, we face two major differences. The first difference is that we are dealing with the heavy-tailed matrices, and the operator norm of our $X$ and its minors are no longer bounded by $C\sqrt{n}$. The second difference is that it is no longer easy to rely on the eigenvalue equation for our singular vector. We can certainly consider the covariance matrix $X^*X$ and view $\mathfrak{u}$  as its eigenvalue. But the dependence structure of $XX^*$ is now more complicated than those models considered in \cite{RV16}. If we use the singular value equation $X\mathfrak{u}=s_{\min}(X)\mathfrak{v}$ and maintain the independence of the entries in the matrix, i.e., $X$, we will also need to involve the left singular vector $\mathfrak{v}$ into the discussion. Currently, it is unclear how to handle $\mathfrak{v}$. Nevertheless, the basic strategy  of  reducing the delocalization problem to smallest singular value problem in \cite{RV16} can still be applied if one aims at a weaker result, i.e, Theorem \ref{thm: 133}, instead of the original no-gaps delocalization in \cite{RV16}. In the sequel, we provide a rough outline of the proof strategy. 

We need to bound the probability of the event
\begin{equation*}
    \{ \exists I \subset [n] : |I| = (1 - \epsilon) n, \lVert \mathfrak{u}_{I} \rVert_2 < \delta_{\mathsf{a}} \}.
\end{equation*}
 For any subset $I \subset [n]$, we start from
$
    s_{\min}(X) = \lVert X \mathfrak{u} \rVert_{2} \ge \lVert X_{I^{c}} \mathfrak{u}_{I^{c}} \rVert_{2} - \lVert X_{I} \mathfrak{u}_{I} \rVert_{2}.
$
The case $\alpha > 4$ is more straightforward to handle. Since $\lVert X_{I} \rVert = O(\sqrt{N})$ and, by the Bai-Yin law, $ |N^{-1/2}s_{\min}(X) - (1 - \sqrt{1/\mathsf{a}})| = o(1)$ with high probability when $\alpha > 4$, we deduce that if $|I|< \delta_{\mathsf{a}}$,
\begin{equation*}
    s_{\min}(X_{I^{c}}) \le \frac{1}{\sqrt{1 - \delta_{\mathsf{a}}^{2}}} (1 - \sqrt{1/\mathsf{a}} + t + C\delta_{\mathsf{a}}) \sqrt{N}, \quad t > 0, \quad C>0.
\end{equation*}
Using the union bound, it reduces to estimating
\begin{equation*}
    {n \choose (1 - \epsilon) n} \cdot \max_{|I| = (1 - \epsilon) n} \prob \left\{ s_{\min}(X_{I^{c}}) \le \frac{(1 - \sqrt{1/\mathsf{a}} + t + C\delta_{\mathsf{a}}) \sqrt{N}}{\sqrt{1 - \delta_{\mathsf{a}}^{2}}} \right\}.
\end{equation*}
Again, by the Bai-Yin law, we observe that $N^{-1/2} |s_{\min}(X_{I^{c}}) - (1 - \sqrt{\epsilon/\mathsf{a}})| = o(1)$ with high probability. Therefore, if $\delta_{\mathsf{a}}$ and $t$ are small enough, we can obtain the desired result by applying a suitable deviation inequality for the smallest singular value, such as \cite[Corollary V.2.1]{FS10}. Note that unlike \cite{RV16} where  the value of the eigenvalue $\lambda$ is irrelevant and the invertibility probability of the almost square matrix $B_{I^c}$ is sufficiently sharp, here we will rely on the precise estimate of the singular values $s_{\min}(X)$ and $s_{\min}(X_{I^c})$ while the large deviation estimate of them is not sharp enough to conclude the original version of the no-gaps delocalization.   

For the case $2 < \alpha \le 4$, a more careful treatment is required since the operator norm of $X_I$ is no longer controlled by $O(\sqrt{N})$. This will be thoroughly explained in Section \ref{sec: 958}, where we address the additional complexities and provide the necessary adjustments.

 Finally, for $\alpha>2$, and any single given index set $I$ with $|I|\geq \epsilon n$, 
one might show that $\mathfrak{u}_{I}$ has a nontrivial contribution to the mass ($\ell^2$-norm) via the idea of the quantum unique ergodicity (QUE) \cite{BY17}. However, the probability bound obtained  from the QUE results is not strong enough to cover all $|I|\geq \epsilon n$ simultaneously, which is required for the no-gaps delocalization.

\section{Localization for $\alpha\in (0,2)$}\label{sec: localization}

Let $X=(x_{ij})$ be as in Definition \ref{def: 78} with $\alpha\in(0,2)$.
Throughout this section, we write $N = \mathsf{a} \cdot n$ to keep it simple, which does not affect the proof. We will use $N$ and $n$ interchangeably in the scaling factors or order estimates, as their roles are interchangeable up to a constant $\mathsf{a}$.
We define $\widetilde{X}=(\tilde{x}_{ij})$ by setting
\begin{equation}\label{eq: 152}
	\widetilde{X} = N^{-\frac{1}{\alpha}+\eps_{N}}X, \quad \eps_{N}\in(0,1/\alpha),
\end{equation}
where we will specify $\eps_{N}$ later.

\begin{definition}\label{def: 174}
Let ${c}>0$ be a constant to be chosen later. 
For any  $v\in\mathbb{S}^{n-1}$, define the set $I\equiv I({c},v,\eps_{N})$ by
\begin{equation}\label{eq: 225}
	I = \{ i \in [n]: |v_{i}|> {c}N^{(-1+\alpha\eps_{N})/2} \}.
\end{equation}
Fix a small constant ${\delta}\in(0,1)$.
We define the subset $\mathcal{S}_{0}\equiv\mathcal{S}_{0}({c},{\delta},\eps_{N})\subset\mathbb{S}^{n-1}$ by
\begin{equation}\label{eq: 207}
	\mathcal{S}_{0} = \{ v\in\mathbb{S}^{n-1} : \lVert v_{I^{c}}\rVert^{2}_{2} > {\delta} \}.
\end{equation}
\end{definition}

We shall focus on finding a lower bound of the following quantity:
\begin{equation}
	\inf_{v\in\mathcal{S}_{0}} \lVert \widetilde{X}v \rVert_{2}. \label{081201}
\end{equation}

\begin{proposition}\label{prop: 240}
Suppose $\alpha\in(0,2)$.
Let $\widetilde{X}=(\tilde{x}_{ij})$ be as in \eqref{eq: 152}.
Let ${c'}>0$ be any small constant and choose $\epsilon_N\equiv \epsilon_N({c'})$ so that $N^{\alpha\eps_{N}}\ge {c'}\log{N}$.
Then, for any ${\delta}\in(0,1)$ and $K>0$, there exist constants ${c},c_{\ref{prop: 240}}>0$ such that
we have for large $N$
\begin{equation*}
	\prob\Big\{ \inf_{v\in\mathcal{S}_{0}} \lVert \widetilde{X}v \rVert_{2} \le c_{\ref{prop: 240}} N^{\alpha\eps_{N}/2} \Big\} = O(N^{-K}),
\end{equation*}
where $\mathcal{S}_{0}\equiv\mathcal{S}_{0}({c},{\delta},\eps_{N})$ is defined as in \eqref{eq: 207}.
In addition, the constant $c_{\ref{prop: 240}}$ may depend on $\mathsf{a}, {\delta}, \alpha$ and $K$, but not on $c'$.
\end{proposition}

A standard approach for obtaining the above lower bound is an $\epsilon$-net argument \cite{RV08}. There exists a sufficiently dense discrete set of $v$'s, denoted by $\mathcal{N}$, such that all vectors in $\mathcal{S}_{0}$ can be approximated by some vector in $\mathcal{N}$, up to a small $\ell^2$ distance $\epsilon$. Consequently, one has $\inf_{v \in \mathcal{S}_{0}} \lVert \widetilde{X}v \rVert_{2} \geq \inf_{v \in \mathcal{N}} \lVert \widetilde{X}v \rVert_{2} - \epsilon \lVert \widetilde{X} \rVert$. Such a strategy would work well if $\lVert \widetilde{X} \rVert$ were well bounded, which is unfortunately not true in the heavy-tailed case. To circumvent this issue, Tikhomirov proposed the following proposition where the operator norm of the heavy-tailed matrix is not involved in the lower bound.

\begin{proposition}[Proposition 3 in \cite{Tikhomirov16}]\label{prop: 153}
Let $N,n\in\N$ and $\mathcal{S}\subset\mathbb{S}^{n-1}$.
Set $Y= Y_{1}+Y_{2}$ where $Y_{1}$ and $Y_{2}$ are $N\times n$ matrices.
Choose ${\delta'}>0$ and $L>0$.
Suppose that there exists a subset $\mathcal{N}\subset\mathbb{R}^{n}$ such that the following conditions hold.
\begin{enumerate}
    \item[(i)] For any $v\in \mathcal{S}$, there exists $v'\in\mathcal{N}$ such that
    \begin{equation*}
        \lVert v|_{\textnormal{supp}(v')} - v' \rVert_{2} \le {\delta'}.
    \end{equation*}
    \item[(ii)] 
 We have for any $v'\in\mathcal{N}$, 
    \begin{equation*}
        \mathrm{d}\big(Y_{1}v',Y(E_{v'}^{\perp})+Y_{2}(E_{v'})\big) \ge L,
    \end{equation*}
    where $E_{v'}=\textnormal{span}\{e_{j}\} _{j\in\textnormal{supp}(v')}$, the linear span of a subset $\{e_{j}\} _{j\in\textnormal{supp}(v')}$ of the standard unit basis in $\mathbb{R}^{n}$.
\end{enumerate}
Then,
\begin{equation}\label{eq: 482}
    \inf_{v\in \mathcal{S}}\lVert Yv \rVert_2 \ge L - {\delta'} \lVert Y_{1} \rVert.
\end{equation}
\end{proposition}
\begin{remark}
The proposition above can be regarded as a special case of \cite[Proposition 3]{Tikhomirov16} in the sense that one might choose more general  collection of linear subspaces $E_{v'}$'s.
\end{remark}

~

When one applies Tikhomirov's proposition to a heavy-tailed matrix $Y$, $Y_1$ is chosen to be a light-tailed part with sufficient anti-concentration properties, while $Y_2$ is the remaining part containing the heavy tail. The advantage of the above proposition lies in the following three aspects:
First, only the operator norm of the light-tailed matrix, $\|Y_1\|$, is involved in the lower bound in (\ref{eq: 482}).
Second, the definition of $\mathcal{S}_0$ allows one to find a net $\mathcal{N}$ in which all vectors $v'$ have sufficiently long length ($\ell^2$-norm) but well-controlled $\ell^{\infty}$-norm, which, together with the anti-concentration nature of $Y_1$, guarantees a sufficiently rich anti-concentration structure of $Y_1v'$.
Third, the conditional independence between $Y_{1}v'$ and $Y(E_{v'}^{\perp}) + Y_{2}(E_{v'})$ allows one to use the fact that a random projection of a random vector with a sufficiently rich structure is anti-concentrated \cite{Tikhomirov16}.

The above strategy was used in \cite{Tikhomirov16} to obtain a lower bound $\sqrt{N}$ of the original random matrix $X$ and its additive deformations, without any moment condition. Here, instead, we apply it to the rescaled matrix $\widetilde{X}$ to show a lower bound of (\ref{081201}). According to the strategy in the above proposition, a decomposition of $\widetilde{X}$ will be needed. The light tail part of our $\widetilde{X}$ will be a restriction of $\widetilde{X}$ on a window of order $1$, which results in a sparse matrix with sparsity around the critical scale. Hence, a major modification we shall make, based on Tikhomirov's strategy, is to extend the anti-concentration results from the dense setting to the sparse setting.
More specifically, 
we decompose $\widetilde{X}$ into two parts so that one part has a nice bound for its operator norm. Let us set
\begin{equation}\label{eq: 232}
	D=D_{1}\cup D_{2}, \quad D_{1}=[-1, 1], \quad D_{2}=[-M,-2)\cup(2,M],
\end{equation}
where we choose $M>0$ large enough so that
\begin{equation}\label{eq: 196}
	\frac{C_{\ell}}{2^{\alpha+1}} > \frac{C_{u}}{M^{\alpha}}.
\end{equation}
We consider the following decomposition of $\widetilde{X}$:
\begin{equation}\label{eq: 235}
	\widetilde{X} = \widetilde{X}_{D} + \widetilde{X}_{D^{c}} = ((\tilde{x}_{ij})_{D}) + ((\tilde{x}_{ij})_{D^{c}}),
\end{equation}
where $(\tilde{x}_{ij})_{D}= \tilde{x}_{ij}\indic\{\tilde{x}_{ij}\in D\}$ and $(\tilde{x}_{ij})_{D^{c}}= \tilde{x}_{ij}\indic\{\tilde{x}_{ij}\in D^{c}\}$.
Since $x_{ij}$ is symmetrically distributed and we assume \eqref{eq: 196}, it follows that
\begin{equation}\label{eq: 248}
	\E[(\tilde{x}_{ij})_{D}]=0, \quad
	\prob\{\tilde{x}_{ij}\in D_{2}\}\ge C_{\ell}2^{-(\alpha+1)}N^{-1+\alpha\eps_{N}}, \quad
	\prob\{\tilde{x}_{ij}\in D_{1}\}\ge 1-C_{u}N^{-1+\alpha\eps_{N}}.
\end{equation}
Note that we use Assumption \ref{assump: 89} (symmetric distribution) to obtain the zero mean condition of $(\tilde{x}_{ij})_{D}$ by simply setting $D$ to be symmetric. For more general cases, such as when $\E x_{ij}=0$ without the symmetry assumption, we may need to choose $D$ accordingly.
One observes that $\widetilde{X}_{D}$ resembles a sparse random matrix with a sparsity level $N^{-1+\alpha\epsilon_{N}}$ approximately. Roughly speaking, $(\tilde{x}_{ij})_{D} \sim 1$ with a probability of $N^{-1+\alpha\epsilon_{N}}$, and is  $0$ otherwise.

We bound the operator norm of $\widetilde{X}_{D}$ in the following lemma.

\begin{lemma}\label{lem: 212}
As in Proposition \ref{prop: 240}, let ${c'}>0$ be any small constant and choose $\epsilon_N= \epsilon_N({c'})$ satisfying $N^{\alpha\eps_{N}}\ge c'\log{N}$.
 For every constant $K>0$, there exists a constant $C_{\ref{lem: 212}}=C_{\ref{lem: 212}}(K)>0$ such that
\begin{equation*}
	\prob\{ \lVert \widetilde{X}_{D} \rVert > C_{\ref{lem: 212}}N^{\frac{\alpha\eps_{N}}{2}} \} = O(N^{-K}).
\end{equation*}
where $C_{\ref{lem: 212}}$ may also depend on $C_{u}$ and the aspect ratio $\mathsf{a}$.
\end{lemma}
\begin{proof}
Let $(\widetilde{X}_{D})_{i \cdot}$ be the $i$-th row of $\widetilde{X}_{D}$.
For $q>0$, we have
\begin{equation*}
	\E\lVert (\widetilde{X}_{D})_{i\cdot}\rVert_{2}^{q} = \E\Big(\sum_{j=1}^{n}(\tilde{x}_{ij})_{D}^{2}\Big)^{\frac{q}{2}}
	\le 2^{\frac{q}{2}}\E\Big|\sum_{j=1}^{n}\big((\tilde{x}_{ij})_{D}^{2}-\E (\tilde{x}_{ij})_{D}^{2}\big)\Big|^{\frac{q}{2}} + 2^{\frac{q}{2}}\Big(\sum_{j=1}^{n}\E (\tilde{x}_{ij})_{D}^{2}\Big)^{\frac{q}{2}}.
\end{equation*}
Using Rosenthal's inequality, 
\begin{multline*}
	\E\Big|\sum_{j=1}^{n}\big((\tilde{x}_{ij})_{D}^{2}-\E (\tilde{x}_{ij})_{D}^{2}\big)\Big|^{\frac{q}{2}} \\
	\le C^{q} \bigg( q^{\frac{q}{2}}\Big(\sum_{j=1}^{n}\E|(\tilde{x}_{ij})_{D}^{2}-\E (\tilde{x}_{ij})_{D}^{2}|^{\frac{q}{2}}\Big) + q^{\frac{q}{4}}\Big(\sum_{j=1}^{n}\E|(\tilde{x}_{ij})_{D}^{2}-\E (\tilde{x}_{ij})_{D}^{2}|^{2}\Big)^{\frac{q}{4}} \bigg),
\end{multline*}
for some constant $C>0$.
Since $(\tilde{x}_{ij})_{D}$ is bounded and $\E (\tilde{x}_{ij})_{D}^{2} \asymp N^{-1+\alpha\eps_{N}}$,
\begin{equation*}
	\E\lVert (\widetilde{X}_{D})_{i\cdot}\rVert_{2}^{q}\le C^{q} \big( (N^{\alpha\eps_{N}})^{\frac{q}{2}} + q^{\frac{q}{2}}N + q^{\frac{q}{4}}(N^{\alpha\eps_{N}})^{\frac{q}{4}} \big),
\end{equation*}
where $C>0$ may depend on $C_{u}$ and the aspect ratio $\mathsf{a}$.

Let $(\widetilde{X}_{D})_{\cdot j}$ be the $j$-th column of $\widetilde{X}_{D}$. We can get a similar estimate for $\E\lVert (\widetilde{X}_{D})_{\cdot j}\rVert_{2}^{q}$. Then, applying Proposition \ref{prop: Seginer}, we may write
\begin{equation*}
	\E\lVert \widetilde{X}_{D} \rVert_{2}^{q}\le 2C^{q} N \big( (N^{\alpha\eps_{N}})^{\frac{q}{2}} + q^{\frac{q}{2}}N + q^{\frac{q}{4}}(N^{\alpha\eps_{N}})^{\frac{q}{4}} \big).
\end{equation*}

Recall that $N^{\alpha\eps_{N}}\ge c'\log{N}$. By setting $q=c'\log N$ and choosing $C_{\ref{lem: 212}}>0$ sufficiently large, the desired result follows from the Markov inequality.
\end{proof}

\subsection{Proof of Proposition \ref{prop: 240}}

 In order to prove Proposition \ref{prop: 240}, we will closely follow the strategy in \cite{Tikhomirov16}, particularly by applying Proposition \ref{prop: 153}. However, as we mentioned above, one difference here is: by considering the decomposition \eqref{eq: 235}, we focus on $\widetilde{X}_{D}$, which resembles a sparse random matrix with sparsity of order $\log n/n$. The sparsity of  $\widetilde{X}_{D}$ leads to worse anti-concentration bounds in various steps, necessitating adjustments to resolve the issue.
Let us present the following technical results, which will be used later together with Proposition \ref{prop: 153}.

\begin{lemma}\label{lem: 331}
Let $c>0$ and $c'',\delta,\delta'\in(0,1)$.
Let $\eps_{N}$ be as in \eqref{eq: 152}.
We define $\mathcal{S}_{0}=\mathcal{S}_{0}(c,\delta,\eps_{N})$ as in Definition \ref{def: 174}.
There exists a finite set $\mathcal{N}_{0}$ and a constant $C_{\ref{lem: 331}}>0$ such that the following holds.
\begin{enumerate}
	\item[(i)] For any $v\in\mathcal{S}_{0}$, there is a vector $v'\in\mathcal{N}_{0}$ satisfying $\lVert v|_{\textnormal{supp}(v')} - v' \rVert_{2} \le {\delta'}$.
	\item[(ii)] For every $v'\in\mathcal{N}_{0}$, we have $\lVert v' \rVert_{2} \ge {\delta}\sqrt{{c''}/2}$, and $\lVert v' \rVert_{\infty}\le {c}N^{(-1+\alpha\eps_{N})/2}$ for large $n$.
	\item[(iii)] $|\mathcal{N}_{0}| \le \exp\big(\log(C_{\ref{lem: 331}}/{\delta'}{c''} )\cdot \lceil {c''}n \rceil\big)$.
\end{enumerate}
\end{lemma}
\begin{remark}
This lemma will be applied with $c = c(c'')$ and $\delta' = \delta'(c'')$, treated as functions of $c''$, to prove Proposition \ref{prop: 240}.
\end{remark}
\begin{proof}
Given any $y\in\mathcal{S}_{0}$ and any integer $m\in [n]$, it is easy to see that 
there exists an index  set $J\equiv J(y)\subset [n]$ such that $|J|\le m$, $\lVert y_{J} \rVert_2 \ge {\delta}\lceil n/m\rceil^{-1/2}$ and $\lVert y_{J} \rVert_{\infty}\le {c}N^{(-1+\alpha\eps_{N})/2}$, according to the definition of $\mathcal{S}_0$ in (\ref{eq: 207}).
For each $y\in\mathcal{S}_{0}$, we set $J=J(y)$ be as above, with $m=\lceil {c''}n \rceil$. Define $\mathcal{T}_{0}$ by 
\begin{equation*}
	\mathcal{T}_{0} = \{y_{J}:y\in\mathcal{S}_{0}\}.
\end{equation*}
A vector $v\in\mathbb{R}^{n-1}$ is said to be $l$-sparse if $|\text{supp}(v)|\le l$. We notice that $\mathcal{T}_{0}$ consists of $m$-sparse vectors satisfying the property (ii).

Then, applying a standard volumetric argument (e.g.~\cite[Eq.~(4.5)]{BR19} and \cite[Lemma 12]{Tikhomirov16}), we find a finite set $\mathcal{N}_{0}\subset \mathcal{T}_{0}$ satisfying the other desired properties.

\end{proof}

\begin{proposition}\label{prop: 335}
Consider any constant ${c''}\in(0,1)$.
Let $C_{\ref{thm: 555}}$ be as in Theorem \ref{thm: 555}.
Fix ${\delta}\in(0,1)$.
Choose ${c}$ by setting
\begin{equation}\label{eq: 535}
	\frac{6C_{\ref{thm: 555}}{c}}{{\delta}{(c'')}^{1/2}C_{\ell}^{1/2}2^{-(\alpha+1)/2}} = 1-\mathsf{a}^{-1/4}.
\end{equation}
Consider a vector $v'$ satisfying $\lVert v' \rVert_{2}\ge \delta\sqrt{c''/2}$ and $\lVert v' \rVert_{\infty}\le cN^{(-1+\alpha\eps_{N})/2}$.
There exist constants $c_{\ref{prop: 335}}>0$ and $c'_{\ref{prop: 335}}>0$ such that
\begin{equation}\label{eq: 653}
	\prob\big\{ \mathrm{d}\big(\widetilde{X}_{D}v',\widetilde{X}(E_{v'}^{\perp})+\widetilde{X}_{D^{c}}(E_{v'})
	\big)\le c_{\ref{prop: 335}}N^{\alpha\eps_{N}/2} \big\} \le e^{-c'_{\ref{prop: 335}}N}.
\end{equation}
In addition, the constant $c'_{\ref{prop: 335}}$ only depends on $\mathsf{a}$.
\end{proposition}
\begin{proof}
Recalling \eqref{eq: 232}, we have $\mathrm{d}(D_{1},D_{2})\ge 1$ and \eqref{eq: 248}.
To derive a concentration inequality, we introduce a copy of $\widetilde{X}$ as follows. 
Let $\widetilde{X}'=(\tilde{x}_{ij}')$ be an $N\times n$ random matrix having the same distribution as $\widetilde{X}$ such that $\tilde{x}_{ij}$ and $\tilde{x}_{ij}'$ are conditionally i.i.d.~given event $\{\tilde{x}_{ij}\in D\}$ and identical on $\{\tilde{x}_{ij}\in D^{c}\}$.

We can observe that for large $n$
\begin{equation*}
    \prob\{|(\tilde{x}_{ij})_{D}-(\tilde{x}_{ij}')_{D}|\ge 1\} \ge \prob\{\tilde{x}_{ij}\in D_{1}; \tilde{x}_{ij}'\in D_{2}\} + \prob\{\tilde{x}_{ij}\in D_{2}; \tilde{x}_{ij}'\in D_{1}\} \ge C_{\ell}2^{-(\alpha+1)}N^{-1+\alpha\eps_{N}}.
\end{equation*}
Recall the definition of the concentration function from (\ref{eq: 345}). 
By symmetry, we notice that 
\begin{equation}\label{eq: 515}
	\mathcal{Q}\big((\tilde{x}_{ij})_{D}-(\tilde{x}_{ij}')_{D},1/3\big)\le 1 - C_{\ell}2^{-(\alpha+1)}N^{-1+\alpha\eps_{N}}.
\end{equation}

By assumption, we have $\lVert v' \rVert_{2}\ge \delta\sqrt{c''/2}$ and $\lVert v' \rVert_{\infty}\le cN^{(-1+\alpha\eps_{N})/2}$.
We write $p=N^{-1+\alpha\eps_{N}}$. Using Theorem \ref{thm: 555}, we have
\begin{align*}
	\prob\bigg\{ \frac{1}{\sqrt{p}}\Big|\sum_{j=1}^{n}((\tilde{x}_{ij})_{D}-(\tilde{x}'_{ij})_{D})v'_{j}\Big|\le {c} \bigg\}
	&\le \mathcal{Q}\Big(\sum_{j=1}^{n}((\tilde{x}_{ij})_{D}-(\tilde{x}'_{ij})_{D})v'_{j},{c}\sqrt{p}\Big) \nonumber\\
	&\hspace{-15ex}\le C_{\ref{thm: 555}} \cdot {c}\sqrt{p} \cdot \bigg(\frac{1}{9}\sum_{j=1}^{n}\Big(1-\mathcal{Q}\big( ((\tilde{x}_{ij})_{D}-(\tilde{x}'_{ij})_{D})v'_{j},|v'_{j}|/3 \big)\Big) (v'_{j})^{2} \bigg)^{-1/2} \nonumber\\
	&\hspace{-15ex} \le \frac{6C_{\ref{thm: 555}}{c}}{{\delta}{(c'')}^{1/2}\sqrt{C_{\ell}2^{-(\alpha+1)}}},
\end{align*}
where we used \eqref{eq: 515} for the last inequality.
Due to the assumption \eqref{eq: 535}, we have
\begin{equation*}
	\prob\bigg\{ \frac{1}{\sqrt{p}}\Big|\sum_{j=1}^{n}((\tilde{x}_{ij})_{D}-(\tilde{x}'_{ij})_{D})v'_{j}\Big|\le {c} \bigg\}
	\le 1-\mathsf{a}^{-1/4}, \quad i\in [N].
\end{equation*}
We can complete the proof using the following lemma, 
whose proof is the same as Lemma 9 of \cite{Tikhomirov16}, and thus is omitted. 

\begin{lemma}\label{lem: 565}
Set $p=N^{-1+\alpha\eps_{N}}$.
Consider $v=(v_{1},v_{2},\cdots,v_{n})\in\R^{n}$ and $s>0$ such that
\begin{equation}\label{eq: 272}
    \prob\bigg\{\bigg| \frac{1}{\sqrt{p}}\sum_{j=1}^{N} \big( (\tilde{x}_{ij})_{D} - (\tilde{x}_{ij}')_{D} \big)v_{j} \bigg| > s\bigg\}
    \ge \mathsf{a} ^{-1/4}, \quad i\in [N].
\end{equation}
For a subset $\mathcal{P}\subset [N]\times [n]$, define
\begin{equation}\label{eq: 333}
	{\mathcal{E}_{\mathcal{P}} = \{ \text{$\tilde{x}_{ij}\in D$ for all $(i,j)\in\mathcal{P}$ and $\tilde{x}_{ij}\in D^{c}$ for all $(i,j)\in\mathcal{P}^{c}$} \}.}
\end{equation}
Define $\mathcal{C}$ as the collection of all subsets $\mathcal{P}$ satisfying
\begin{equation*}
    \prob(\mathcal{E}_{\mathcal{P}})>0 \quad \text{and} \quad 
    \bigg| \bigg\{ i\in\{1,2,\cdots,N\}: \mathcal{Q}_{\mathcal{E}_{\mathcal{P}}}\bigg(\frac{1}{\sqrt{p}}\sum_{j=1}^{n}(\tilde{x}_{ij})_{D}v_{j},\frac{s}{2}\bigg)\le 1-\tau \bigg\} \bigg| \ge N\mathsf{a}^{-1/2},
\end{equation*}
where $\tau=\frac{1}{2}(\mathsf{a}^{-1/4}-\mathsf{a}^{-1/3})$ and for an event $\mathcal{E}$, 
the function $\mathcal{Q}_{\mathcal{E}}$ represents the concentration function conditioned on $\mathcal{E}$, defined as:
\begin{equation*}
\mathcal{Q}_{\mathcal{E}}(\cdot,t) = \sup_{\nu\in\mathbb{R}}\prob(\lVert \cdot-\nu \rVert_2\le t \mid \mathcal{E}),
\quad t > 0.
\end{equation*}
Then, there exists a constant $c_{\ref{lem: 565}}>0$ only depending on $\mathsf{a}$ such that
\begin{equation*}
    \prob\bigg( \bigcup_{\mathcal{P}\in\mathcal{C}}\mathcal{E}_{\mathcal{P}} \bigg) \ge 1 - \exp(-c_{\ref{lem: 565}}N).
\end{equation*}
\end{lemma}

Let $\mathcal{C}$ and $\tau$ be as in Lemma \ref{lem: 565} with $v=v'$ and $s={c}$. Take $\mathcal{P}\in\mathcal{C}$.
Set
\begin{equation*}
	\mu = \bigg| \bigg\{ i\in [N]: \mathcal{Q}_{\mathcal{E}_{\mathcal{P}}}\bigg(\frac{1}{\sqrt{p}}\sum_{j=1}^{n}(\tilde{x}_{ij})_{D}v'_{j},\frac{{c}}{2}\bigg)\le 1-\tau \bigg\} \bigg|.
\end{equation*}
Since we choose $\mathcal{P}\in\mathcal{C}$, we note that $\mu \ge N\mathsf{a}^{-1/2}$. 
Using Lemma \ref{lem: 1081} with $d\ge\mu-n$ and $\ell=e\cdot C_{\ref{lem: 1081}}^{2}\cdot\tau^{-1}$,
for $\kappa=\mathsf{a}^{-1/2}-\mathsf{a}^{-1}$ and an $n$-dimensional subspace $F\subset\mathbb{R}^{N}$, we obtain
\begin{equation}\label{eq: 774}
	\prob_{\mathcal{E}_{\mathcal{P}}}\{ \mathrm{d}(\widetilde{X}_{D}v',F) \le \frac{{c}}{2\ell}\sqrt{\kappa pN} \} \le e^{-\kappa N/2\ell},
\end{equation}
where we use the fact that $\mathrm{d}(\widetilde{X}_{D}v',F) = \textnormal{Proj}_{F^{\perp}}(\widetilde{X}_{D}v')$ for the orthogonal complement $F^{\perp}$ of the subspace $F$.

Noticing that $\widetilde{X}_{D}v'$ and $\widetilde{X}(E_{v'}^{\perp}) + \widetilde{X}_{D^{c}}(E_{v'})$ are conditionally independent given the event $\mathcal{E}_{\mathcal{P}}$, the desired result follows from Lemma \ref{lem: 565}.
\end{proof}

\begin{proof}[Proof of Proposition \ref{prop: 240}]
Let us set $c$ as in \eqref{eq: 535}.
Recall the definitions of $\kappa, \ell, \tau$ and $p$ in the proof of Proposition \ref{prop: 335}. 
By \eqref{eq: 535} and \eqref{eq: 774}, we define $c_{\ref{prop: 335}}$ as
\begin{equation*}
	c_{\ref{prop: 335}} = (1-\mathsf{a}^{-1/4}) \delta(c'')^{1/2} C_{\ell}^{1/2} 2^{-(\alpha+5)/2} 6^{-1} C_{\ref{thm: 555}}^{-1} e^{-1} C_{\ref{lem: 1081}}^{-2} (\mathsf{a}^{-1/4}-\mathsf{a}^{-1/3}) (\mathsf{a}^{-1/2}-\mathsf{a}^{-1})^{1/2}.
\end{equation*}

For a constant $K>0$, let $C_{\ref{lem: 212}}(K)$ be as in Lemma \ref{lem: 212}.
To apply a standard net argument along with Lemma \ref{lem: 331}, we need to set $c''$ and $\delta'$ appropriately. We define $\delta'$ by setting
\begin{equation}\label{eq: 791}
	\delta' = \frac{c_{\ref{prop: 335}}}{2C_{\ref{lem: 212}}(K)}.
\end{equation}
(Note that $\delta'$ depends on $c''$ as defined above.) Next, recalling that $c'_{\ref{prop: 335}}$ only depends on $\mathsf{a}$, we choose $c''>0$ sufficiently small such that
\begin{equation}\label{eq: 777}
c''\log(C_{\ref{lem: 331}}/\delta' c'') < \frac{c'_{\ref{prop: 335}}}{4}.	
\end{equation}

Let $\mathcal{N}_{0}$ be the net as in Lemma \ref{lem: 331}. For any $v\in\mathcal{S}_{0}$, we can find $v'\in\mathcal{N}_{0}$ such that  $\lVert v|_{\textnormal{supp}(v')} - v' \rVert_{2} \le \delta'$. In addition, since this $v'$ satisfies $\lVert v' \rVert_{2} \ge \delta\sqrt{c''/2}$ and $\lVert v' \rVert_{\infty}\le cN^{(-1+\alpha\epsilon_{N})/2}$ for large $n$, we can apply Proposition \ref{prop: 335} to obtain the estimate \eqref{eq: 653}. Due to our choice of $c''$ in \eqref{eq: 777}, we observe that
\begin{equation*}
	|\mathcal{N}_{0}| \cdot \max_{v'\in\mathcal{N}_{0}} \prob\big\{ \mathrm{d}\big(\widetilde{X}_{D}v',\widetilde{X}(E_{v'}^{\perp})+\widetilde{X}_{D^{c}}(E_{v'})
	\big)\le c_{\ref{prop: 335}}N^{\alpha\epsilon_{N}/2} \big\} \le e^{-c'_{\ref{prop: 335}}N/2}.
\end{equation*}

Now, we are ready to apply Proposition \ref{prop: 153}. Together with Lemma \ref{lem: 212} (upper bound of $\lVert \widetilde{X}_{D}\rVert$), we complete the proof.
\end{proof}

\subsection{Upper bound of the smallest singular value}

\begin{theorem}\label{thm: upper bound}
Suppose $\alpha\in(0,2)$.
For any $K>0$, there exist a constant $C_{\ref{thm: upper bound}}>0$ such that 
\begin{equation*}
	\prob\{s_{\min}(X) \ge C_{\ref{thm: upper bound}}N^{\frac{1}{\alpha}}(\log N)^{\frac{\alpha-2}{2\alpha}}\}= O(N^{-K}).
\end{equation*}
In addition, the constant $C_{\ref{thm: upper bound}}$ may depend on $K$ and $\mathsf{a}$.
\end{theorem}

\begin{proof}

Suppose $\tilde{\eps}_{N}\in(0,1/\alpha)$, whose precise value will be chosen later.
Let $\Psi=(\psi_{ij})$ be an $N\times n$ random matrix such that all $\psi_{ij}$'s are independent Bernoulli random variables defined by
\begin{equation}\label{eq: 863}
	\psi_{ij} =
	\begin{cases}
		1, & \text{with probability $\prob\{|x_{ij}|\le N^{\frac{1}{\alpha}-\tilde{\eps}_{N}}\}$},\\
		0, & \text{with probability $\prob\{|x_{ij}|> N^{\frac{1}{\alpha}-\tilde{\eps}_{N}}\}$}.
	\end{cases}
\end{equation}
For each $(i,j)\in [N]\times [n]$, let $y_{ij}$ and $z_{ij}$ be independent random variables such that
\begin{align}
	\prob\{y_{ij}\in \tilde{I}\} &= \frac{ \prob\{x_{ij}\in[-N^{\frac{1}{\alpha}-\tilde{\eps}_{N}},N^{\frac{1}{\alpha}-\tilde{\eps}_{N}}]\cap \tilde{I}\} }{ \prob\{|x_{ij}|\le N^{\frac{1}{\alpha}-\tilde{\eps}_{N}}\} },\\
	\prob\{z_{ij}\in \tilde{I}\} &= \frac{ \prob\{x_{ij}\in((-\infty,-N^{\frac{1}{\alpha}-\tilde{\eps}_{N}})\cup(N^{\frac{1}{\alpha}-\tilde{\eps}_{N}},\infty))\cap \tilde{I}\} }{ \prob\{|x_{ij}|> N^{\frac{1}{\alpha}-\tilde{\eps}_{N}}\} }, \label{081410}
\end{align}
for every interval $\tilde{I}\subset\mathbb{R}$.
Furthermore, we assume that ${\psi_{ij}}$, ${y_{ij}}$, and ${z_{ij}}$ are independent. Then,  $\{x_{ij}\}$ has the same law as $\{y_{ij}\psi_{ij} + z_{ij}(1-\psi_{ij})\}$, i.e.,
\begin{equation*}
	X \overset{d}{=} \Psi \circ Y + (\mathbf{1}_{N,n}-\Psi) \circ Z,
\end{equation*}
where $Y=(y_{ij})$, $Z=(z_{ij})$ and $\mathbf{1}_{N,n}\in\mathbb{R}^{N\times n}$ is the all-one matrix.
This form of decomposition/resampling was previously used in \cite{Aggarwal19, ALY}.

Let $\Psi_{j}=(\psi_{1j}, \psi_{2j}, \cdots, \psi_{Nj})^{\mathsf{T}}$ be $j$-th column of $\Psi$.
We denote by $\mathbf{1}_{N}\in\mathbb{R}^{N}$ the  all-one vector.
For any constant $\mathfrak{a}>1$, we have
\begin{equation*}
	\exp(-\mathfrak{a}\cdot C_{u}N^{\alpha\tilde{\eps}_{N}})\le \prob\{ \Psi_{j} = \mathbf{1}_{N} \} \le \exp(-C_{\ell}N^{\alpha\tilde{\eps}_{N}}),
	\quad j\in [n].
\end{equation*}
Let $\mathfrak{m}$ be the number of all-one columns in $\Psi=(\Psi_{1},\Psi_{2},\cdots, \Psi_{n})$.
We observe that
\begin{equation}\label{eq: 231}
	\prob\{\mathfrak{m}=0\} \le \exp\big(-n\exp(-\mathfrak{a}\cdot C_{u}N^{\alpha\tilde{\eps}_{N}})\big).
\end{equation}
Let $Y^{(\mathfrak{m})}$ be the minor of $Y$ obtained by removing all columns with indices in $\{j: \Psi_{j}\neq\mathbf{1}_{N}\}$.
Fix any $K>0$.
Since
\begin{equation*}
	s_{\min}(\Psi Y + (\mathbf{1}_{N,n}-\Psi) Z) = \inf_{v\in\mathbb{S}^{n-1}} \lVert (\Psi \circ Y + (\mathbf{1}_{N,n}-\Psi)\circ  Z)v \rVert \le \lVert Y^{(\mathfrak{m})} \rVert,
\end{equation*}
using \eqref{eq: 231} and applying Proposition \ref{prop: Seginer} as in Lemma \ref{lem: 212}, we have for any $K>0$
\begin{equation*}
	\prob\{ \lVert Y^{(\mathfrak{m})}\rVert > C N^{\frac{1}{\alpha}-\alpha\tilde{\eps}_{N}(\frac{2-\alpha}{2\alpha})} \}
	\le N^{-K} + \exp\big(-n\exp(-\mathfrak{a}\cdot C_{u}N^{\alpha\tilde{\eps}_{N}})\big),
\end{equation*}
where $C$ depends on $K$, $C_{u}$ and $\mathsf{a}$.

For a constant $\mathfrak{b}\in(0,1)$, define $\tilde{\eps}_{N}$ through
\begin{equation}\label{eq: 721}
	N^{\alpha\tilde{\eps}_{N}} = \frac{\mathfrak{b}\log{N}}{\mathfrak{a}\cdot C_{u}}.
\end{equation}
Then, the desired result follows.
\end{proof}

\begin{remark} In the above proof, we used the minor matrix $Y^{(\mathfrak{m})}$, which, in distribution, is the same as the minor of $X$ with columns whose entries are all bounded by $N^{\frac{1}{\alpha}-\tilde{\eps}_{N}}$. These columns have small lengths.
 Here, we used the operator norm of $Y^{(\mathfrak{m})}$, i.e., the largest singular value, to provide an upper bound for $s_{\min}(X)$ for simplicity. Notice that, by Cauchy interlacing, one can actually use the smallest singular value of $Y^{(\mathfrak{m})}$ to serve as an upper bound. A further natural question is whether the smallest singular value of $Y^{(\mathfrak{m})}$ exactly matches $s_{\min}(X)$ in order, i.e., not only as an upper bound. We leave this as a further study.
\end{remark}

In the sequel, we also denote by $s_{n-k}(X)$ the $k$-th smallest singular value of $X$. In particular, $s_{\min}(X)=s_n(X)$. 

\begin{corollary}\label{cor: 729}
Suppose $\alpha\in(0,2)$.
Let $\mathfrak{b}\in(0,1/2)$ be a constant. Consider $k=O(n^{1-2\mathfrak{b}})$.
For any $K>0$, there exist constants $C_{\ref{cor: 729}}>0$ such that 
\begin{equation*}
	\prob\{s_{n-k}(X) \ge C_{\ref{cor: 729}}\cdot N^{\frac{1}{\alpha}}(\log N)^{\frac{\alpha-2}{2\alpha}}\}= O(N^{-K}).
\end{equation*}
In addition, the constant $C_{\ref{cor: 729}}$ may depend on $K$, $C_{u}$ and $\mathsf{a}$.
\end{corollary}
\begin{proof}
For a constant $\mathfrak{a}>1$, set $\tilde{\eps}_{N}$ as in \eqref{eq: 721}.
Let $\mathfrak{m}$ be as in the proof of Theorem \ref{thm: upper bound}.
By Hoeffding's inequality, we have $\mathfrak{m} > n^{1-\mathfrak{b}}/2$ with probability at least $1-\exp(-n^{1-2\mathfrak{b}}/2)$.
According to the min-max theorem, it is known that
\begin{equation*}
	s_{n-k} = \max_{\textnormal{dim}(S)=n-k}\min_{\substack{v\in S\\ \lVert v\rVert_{2}=1}}\lVert Xv\rVert_2.
\end{equation*}
Note that $(n-k)+\mathfrak{m}>n$ on the event $\{\mathfrak{m} > n^{1-\mathfrak{b}}/2\}$. Define the subset $\mathcal{J}\subset[n]$ by
\begin{equation*}
	\mathcal{J} = \{j\in[n]: \Psi_{j}=\mathbf{1}_{N}\}.
\end{equation*}
By definition, we observe that $\mathfrak{m}=|\mathcal{J}|$. Conditioning on $\Psi$ such that $\mathfrak{m} > n^{1-\mathfrak{b}}/2$ holds, for any subspace $S$ of $\mathbb{R}^{n}$ with $\textnormal{dim}(S)=n-k$, since $(n-k)+\mathfrak{m}>n$, we find that
\begin{equation*}
	S\cap\,\textnormal{span}\{e_{j}: j\in\mathcal{J}\}\neq\emptyset,
\end{equation*}
which implies
\begin{equation*}
	\min_{\substack{v\in S\\ \lVert v\rVert_{2}=1}} \lVert (\Psi \circ Y + (\mathbf{1}_{N,n}-\Psi)\circ Z)v \rVert_2 \le \lVert Y^{(\mathfrak{m})} \rVert.
\end{equation*}
Now we can complete the proof by following the argument of Theorem \ref{thm: upper bound}.
\end{proof}

\subsection{Proofs of Theorem \ref{thm: ssv} and Theorem \ref{thm: 220}}

We shall prove the localization result combining Proposition \ref{prop: 240} and Theorem \ref{thm: upper bound}.

\begin{proof}[Proof of Theorem \ref{thm: ssv}]
Let ${c}$ be as in Proposition \ref{prop: 240}.
Consider a constant $c'>0$.
Set $N^{\alpha\eps_{N}}={c'}\log{N}$ and $c_{\ref{thm: ssv}}={c}^{2}\cdot {c'}\cdot \mathsf{a}^{-1}$.
Then, by Definition \ref{def: 174}, it is enough to show that
\begin{equation}\label{eq: 803}
	\prob\{\mathfrak{u}\in\mathcal{S}_{0}\} = O(N^{-K}).
\end{equation}
Fix $C>C_{\ref{thm: upper bound}}$. Let $c_{\ref{prop: 240}}$ be as in Proposition \ref{prop: 240}.
Choose ${c'}>0$ sufficiently small such that $c_{\ref{prop: 240}}\cdot ({c'})^{\frac{\alpha-2}{2\alpha}}\ge C$.
Applying Proposition \ref{prop: 240} with $N^{\alpha\eps_{N}}={c'}\log{N}$, we have
\begin{equation}
	\prob\Big\{ \inf_{v\in\mathcal{S}_{0}} \lVert X v \rVert_2 \le C N^{\frac{1}{\alpha}}(\log N)^{\frac{\alpha-2}{2\alpha}} \Big\} = O(N^{-K}). \label{081401}
\end{equation}
Using the fact that $s_{\min}(X)=\lVert X \mathfrak{u} \rVert_2$,
we also find that the event
\begin{equation*}
	\{\mathfrak{u}\in\mathcal{S}_{0}\}\cap \Big\{ \inf_{v\in\mathcal{S}_{0}} \lVert X v \rVert_2 > C N^{\frac{1}{\alpha}}(\log N)^{\frac{\alpha-2}{2\alpha}} \Big\}
\end{equation*}
implies
\begin{equation*}
	s_{\min}(X) > C N^{\frac{1}{\alpha}}(\log N)^{\frac{\alpha-2}{2\alpha}}.
\end{equation*}
Recalling that $C>C_{\ref{thm: upper bound}}$, we notice that
\begin{equation*}
	\{\mathfrak{u}\in\mathcal{S}_{0}\}\cap \Big\{ \inf_{v\in\mathcal{S}_{0}} \lVert X v \rVert_2 > C N^{\frac{1}{\alpha}}(\log N)^{\frac{\alpha-2}{2\alpha}} \Big\}
	\cap \{ s_{\min}(X) < C_{\ref{thm: upper bound}} N^{\frac{1}{\alpha}}(\log N)^{\frac{\alpha-2}{2\alpha}} \}
	= \emptyset.
\end{equation*}
Note that
\begin{equation*}
	\prob\{\mathfrak{u}\in\mathcal{S}_{0}\} \le \prob\Big\{ \inf_{v\in\mathcal{S}_{0}} \lVert X v \rVert_2 \le C N^{\frac{1}{\alpha}}(\log N)^{\frac{\alpha-2}{2\alpha}} \Big\}
	+ \prob\{ s_{\min}(X) \ge C_{\ref{thm: upper bound}} N^{\frac{1}{\alpha}}(\log N)^{\frac{\alpha-2}{2\alpha}} \}.
\end{equation*}
We can complete the proof of  \eqref{eq: 803} by combining equation (\ref{081401}) with Theorem \ref{thm: upper bound}.
\end{proof}

\begin{remark}[Localization length]
Consider the size of the set $\hat{I}=\hat{I}(\mathfrak{u})$, which is defined in Theorem \ref{thm: ssv}. By Markov's inequality, we have
\begin{equation}\label{eq: 885}
	|\hat{I}(\mathfrak{u})| \le \frac{n}{c_{\ref{thm: ssv}}\log{n}}.
\end{equation}
Recall that we set $c_{\ref{thm: ssv}}={c}^{2}\cdot {c'}$ in the proof. The constants ${c}$ and ${c'}$ satisfy
\eqref{eq: 535} and $c_{\ref{prop: 240}}\cdot ({c'})^{\frac{\alpha-2}{2\alpha}}\ge C$ with $C>C_{\ref{thm: upper bound}}$.
Thus we get
\begin{equation*}
	{c'} \le (c_{\ref{prop: 240}}/C_{\ref{thm: upper bound}})^{\frac{2\alpha}{2-\alpha}},
\end{equation*}
which implies the constant $c_{\ref{thm: ssv}}$ should be very small when we consider $\alpha$ close to $2$.
This may indicate that the vector becomes more delocalized as $\alpha$ approaches $2$, which aligns with our intuition. Moreover, for each $\alpha\in(0,2)$, we may find $c_{\ref{thm: ssv}}$ such that $c_{\ref{thm: ssv}}$ is decreasing in $\alpha$.
\end{remark}

\begin{proof}[Proof of Theorem \ref{thm: 220}]
This can be regarded as an extension to the $k$-th smallest singular value $s_{n-k}(X)$.
The proof is similar.
Let $\mathfrak{u}^{(k)}=(\mathfrak{u}^{(k)}_{1},\mathfrak{u}^{(k)}_{2},\cdots,\mathfrak{u}^{(k)}_{n})$ be  unit.
Since we have $s_{n-k}=\lVert X\mathfrak{u}^{(k)} \rVert_2$, the event
\begin{equation*}
	\{\mathfrak{u}^{(k)}\in\mathcal{S}_{0}\}\cap \Big\{ \inf_{v\in\mathcal{S}_{0}} \lVert X v \rVert_2 > C N^{\frac{1}{\alpha}}(\log N)^{\frac{\alpha-2}{2\alpha}} \Big\}
\end{equation*}
implies
\begin{equation*}
	s_{n-k}(X) > C N^{\frac{1}{\alpha}}(\log N)^{\frac{\alpha-2}{2\alpha}}.
\end{equation*}
Let $C_{\ref{cor: 729}}$ be as in Corollary \ref{cor: 729}.
Choose $C>C_{\ref{cor: 729}}$. By Corollary \ref{cor: 729}, it follows that
\begin{equation*}
	\prob\{s_{n-k} \le C_{\ref{cor: 729}} N^{\frac{1}{\alpha}}(\log N)^{\frac{\alpha-2}{2\alpha}}\} = 1 - O(N^{-K}).
\end{equation*}
The remaining argument is almost the same as the proof of Theorem \ref{thm: ssv}, and thus we omit it.
\end{proof}

\section{Delocalization for $\alpha>2$}\label{sec: 958} 

In this section, we shall prove Theorem \ref{thm: 133} for the case $\alpha>2$.
As in the previous section, for simplicity, let us write $N = \mathsf{a} \cdot n$ and use $N$ and $n$ interchangeably. This adjustment does not affect the proof.
Before starting the proof, we introduce the following two propositions.

\begin{proposition}\label{prop: 709}
For $\alpha>2$, there exist constants $c_{\ref{prop: 709}},c'_{\ref{prop: 709}}\in(0,1/2)$ only depending on $\alpha$ such that the following holds.
Define the subset $\mathcal{D}\subseteq [N]\times [n]$ through 
\begin{equation*}
	\mathcal{D} = \{(i,j)\in [N]\times [n]: |x_{ij}| > N^{\frac{1}{2}-c_{\ref{prop: 709}}}\}.
\end{equation*}
Then we have for large $n$
\begin{equation*}
	\prob\{ |\mathcal{D}| > N^{1-c'_{\ref{prop: 709}}} \} \le N^{- \frac{(\alpha-2)}{4} N^{1-c'_{\ref{prop: 709}}}}.
\end{equation*}
\end{proposition}
\begin{proof}
Since $\prob\{|x_{ij}|>N^{\frac{1}{2}-c_{\ref{prop: 709}}}\}\le C_{u}N^{-\frac{\alpha}{2}+\alpha c_{\ref{prop: 709}}}$,
we have
\begin{align}\label{eq: 1227}
	\prob\{ |\mathcal{D}| > N^{1-c'_{\ref{prop: 709}}} \}
	&\le \sum_{ k > N^{1-c'_{\ref{prop: 709}}} } {Nn \choose k}(C_{u}N^{-\frac{\alpha}{2}+\alpha c_{\ref{prop: 709}}})^{k} \nonumber\\
	&\le \sum_{ k > N^{1-c'_{\ref{prop: 709}}} } \Big(\frac{eNn}{k}\Big)^{k}(C_{u}N^{-\frac{\alpha}{2}+\alpha c_{\ref{prop: 709}}})^{k}.
\end{align}
We observe that each term in \eqref{eq: 1227} is bounded by
\begin{equation*}
	N^{-k(\frac{\alpha-2}{2}+c'_{\ref{prop: 709}}+\alpha c_{\ref{prop: 709}})},
\end{equation*}
for $N$ sufficiently large.
Since $\alpha>2$, we can obtain the desired result by choosing $c_{\ref{prop: 709}},c'_{\ref{prop: 709}}>0$ small enough.
\end{proof}

\begin{proposition}\label{prop: 716}
Suppose $\alpha>2$. 
Let $c_{\ref{prop: 709}}$ be as in Proposition \ref{prop: 709}.
Without loss of generality, we assume that $c_{\ref{prop: 709}}>0$ is small enough so that $N^{-1+2c_{\ref{prop: 709}}}\gg \max\{N^{-\frac{\alpha}{2}+\alpha c_{\ref{prop: 709}}}, N^{-2+4c_{\ref{prop: 709}}}\log N\}$. 
Define the (random) subset $J\subset [n]$ by
\begin{equation*}
	J = \{ j\in[n]: \textnormal{$|x_{ij}|\le N^{\frac{1}{2}-c_{\ref{prop: 709}}}$ for all $i$} \}.
\end{equation*}
We denote by $X_{J}$ the minor of $X$ which consists of the columns indexed by $J$. 
For every constant $K>0$, 
there exists a constant $C_{\ref{prop: 716}}=C_{\ref{prop: 716}}(K)>0$ such that
\begin{equation*}
	\prob\{ \lVert X_{J} \rVert > C_{\ref{prop: 716}}\sqrt{N} \} = O(N^{-K}).
\end{equation*}
\end{proposition}
\begin{proof}

Similarly to \eqref{eq: 863} in the proof of Theorem \ref{thm: upper bound},
we let $\Phi=(\phi_{ij})$ be an $N\times n$ random matrix such that all $\phi_{ij}$'s are independent Bernoulli random variables defined by
\begin{equation*}
	\phi_{ij} =
	\begin{cases}
		1, & \text{with probability $\prob\{|x_{ij}|\le N^{\frac{1}{2}-c_{\ref{prop: 709}}}\}$},\\
		0, & \text{with probability $\prob\{|x_{ij}|> N^{\frac{1}{2}-c_{\ref{prop: 709}}}\}$}.
	\end{cases}
\end{equation*}
For each $(i,j)\in\{1,2,\cdots,N\}\times\{1,2,\cdots,n\}$, let $a_{ij}$ and $b_{ij}$ be independent random variables such that
\begin{align*}
	\prob\{a_{ij}\in \tilde{I}\} &= \frac{ \prob\{x_{ij}\in[-N^{\frac{1}{2}-c_{\ref{prop: 709}}},N^{\frac{1}{2}-c_{\ref{prop: 709}}}]\cap \tilde{I}\} }{ \prob\{|x_{ij}|\le N^{\frac{1}{2}-c_{\ref{prop: 709}}}\} },\\
	\prob\{b_{ij}\in \tilde{I}\} &= \frac{ \prob\{x_{ij}\in((-\infty,-N^{\frac{1}{2}-c_{\ref{prop: 709}}})\cup(N^{\frac{1}{2}-c_{\ref{prop: 709}}},\infty))\cap \tilde{I}\} }{ \prob\{|x_{ij}|> N^{\frac{1}{2}-c_{\ref{prop: 709}}}\} },
\end{align*}
for every interval $\tilde{I}\subset\mathbb{R}$.
Furthermore, we assume that ${\phi_{ij}}$, ${a_{ij}}$, and ${b_{ij}}$ are independent. Then,  $\{x_{ij}\}$ has the same law with $\{a_{ij}\phi_{ij} + b_{ij}(1-\phi_{ij})\}$, i.e.,
\begin{equation*}
	X \overset{d}{=} \Phi \circ A + (\mathbf{1}_{N,n}-\Phi) \circ B,
\end{equation*}
where $A=(a_{ij})$, $B=(b_{ij})$ and $\mathbf{1}_{N,n}\in\mathbb{R}^{N\times n}$ is the all-one matrix.
 Let $\widetilde{\mathcal{D}}_{c} \equiv \widetilde{\mathcal{D}}_{c}(\Phi) = \{  j \in [n] : \sum\nolimits_{i \in [N]} \phi_{ij} < N\}$, and denote by  $Y^{(\widetilde{\mathcal{D}}_{c} )}$ the minor of $Y$ obtained by removing all columns indexed by $\widetilde{\mathcal{D}}_{c}$. Then we have $X_J \overset{d}{=} (\Phi \circ A + (\mathbf{1}_{N,n}-\Phi) \circ B)^{(\widetilde{\mathcal{D}}_{c} )}= A^{(\widetilde{\mathcal{D}}_{c} )}$.  Note that  $N^{-1/2} A^{(\widetilde{\mathcal{D}}_{c} )}$ is the so-called random matrix with bounded support. Given any $\Phi$ which has at most $N^{1-c'_{\ref{prop: 709}}}$ entries equal to $0$, it follows from \cite[Lemma 3.11]{DY1} that for any $K > 0$ and sufficiently large $N$, 
\begin{align*}
	\prob_{\Phi}( \lVert A^{(\widetilde{\mathcal{D}}_{c} )} \lVert  \geq 3 \sqrt{N} ) = O(N^{-K}),
\end{align*}
where $\prob_{\Phi}(\cdot)$ represents the conditional probability given $\Phi = (\phi_{ij})$.
Using the law of total probability together with Proposition \ref{prop: 709}, we have with $ \Omega \equiv \{\mathcal{A} \in \{0,1\}^{N \times n}: \sum_{i \in [N],j\in [n]} \mathcal{A}_{ij} \geq nN - N^{1-c'_{\ref{prop: 709}}} \}$,
\begin{align*}
	\prob(\lVert A^{(\widetilde{\mathcal{D}}_{c} )} \lVert  \geq 3 \sqrt{N} ) &= \E \big\{ \prob_{\Phi}(\lVert A^{(\widetilde{\mathcal{D}}_{c} )} \lVert  \geq 3 \sqrt{N} )\cdot  \mathbf{1}_{\Phi \in  \Omega} \big\}\\
	&\quad + \E\big\{ \prob_{\Phi}(\lVert A^{(\widetilde{\mathcal{D}}_{c} )} \lVert  \geq 3 \sqrt{N} )\cdot  (1 -\mathbf{1}_{\Phi \in  \Omega})\big\} = O(N^{-K}),
\end{align*} 
which concludes the proof.
\end{proof}

\begin{proposition}\label{prop: 1029} Suppose $\alpha>2$. 
Fix $t>0$.
For any constant $K>0$, 
we have
\begin{equation*}
	\prob\{ N^{-1/2}s_{\min}(X) \ge 1-\sqrt{1/\mathsf{a}} + t \} = O(N^{-K}).
\end{equation*}
\end{proposition}
\begin{proof} 
 Let $\Phi$ be defined as in the proof of Proposition \ref{prop: 716}. Recall $A,B$ and $\widetilde{\mathcal{D}}_{c}$ also.
Given any $\Phi$ which has at most $N^{1-c'_{\ref{prop: 709}}}$ entries equal to $0$, the Cauchy interlacing gives, for any fixed $t > 0$,
\begin{align*}
	\prob_{\Phi}(s_{\min}( \Phi \circ A + (\mathbf{1}_{N,n}-\Phi) \circ B) \ge 1-\sqrt{1/\mathsf{a}} + t   )\leq \prob_{\Phi}(s_{n - |\widetilde{\mathcal{D}}_{c} |}(  A^{(\widetilde{\mathcal{D}}_{c} )}) \geq 1-\sqrt{1/\mathsf{a}} + t ).
\end{align*} 
As mentioned above,   $N^{-1/2} A^{(\widetilde{\mathcal{D}}_{c} )}$ is the so-called random matrix with bounded support. Adapting the proof of \cite[Theorem 2.9]{HLS} from the right edge to the left edge with a crude lower bound of $s_{n - |\widetilde{\mathcal{D}}_{c} |}(  A^{(\widetilde{\mathcal{D}}_{c} )})$ given in \cite{Tikhomirov16}, we obtain
\begin{align*}
	\prob_{\Phi}(s_{n - |\widetilde{\mathcal{D}}_{c} |}(  A^{(\widetilde{\mathcal{D}}_{c} )}) \geq 1-\sqrt{1/\mathsf{a}} + t) = O(N^{-K}).
\end{align*}
The claim then follows by applying the law of total probability together with Proposition \ref{prop: 709}.
\end{proof}

\begin{proof}[Proof of Theorem \ref{thm: 133}]

Recall $\lVert X \mathfrak{u} \rVert_{2} = s_{\min}(X)$.
For $\eps\in(0,1/2)$, define the event 
\begin{equation*}
	\mathcal{E}_{\textnormal{bad}} = \{ \exists I\subset [n] : |I|=(1-\eps)n, \lVert \mathfrak{u}_{I} \rVert_{2} < \delta_{\mathsf{a}} \},
\end{equation*}
where $\delta_{\mathsf{a}}$ will be specified later.
Let $\mathcal{D}$ and $J$ be as in Proposition \ref{prop: 709} and Proposition \ref{prop: 716} respectively.
We also introduce the events $\mathcal{E}_{1}$, $\mathcal{E}_{2}$ and $\mathcal{E}_{3}$ as follows:
\begin{equation*}
	\mathcal{E}_{1}=\{|\mathcal{D}|\le N^{1-c'_{\ref{prop: 709}}}\},\quad
	\mathcal{E}_{2}=\{\lVert X_{J} \rVert \le C_{\ref{prop: 716}}\sqrt{N}\},\quad
	\mathcal{E}_{3}=\{N^{-1/2}s_{\min}(X) < 1-\sqrt{1/\mathsf{a}} + t\},
\end{equation*}
where we will choose $t>0$ later.
By Proposition \ref{prop: 709}, Proposition \ref{prop: 716} and Proposition \ref{prop: 1029}, for any constant $K>0$, we have
\begin{equation*}
	\prob(\mathcal{E}_{1}^{c}\cup\mathcal{E}_{2}^{c}\cup\mathcal{E}_{3}^{c}) = O(N^{-K}).
\end{equation*}
Consider the event $\mathcal{E}_{\textnormal{bad}} \cap \mathcal{E}_{1} \cap \mathcal{E}_{2} \cap \mathcal{E}_{3}$.
 Recall the notation $X_{L}$ as the minor of $X$ consisting of the columns indexed by $L\subset [n]$.
Since we have for any subset $I\subset [n]$
\begin{equation}\label{eq: 1035}
	\lVert X \mathfrak{u} \rVert_{2} = \lVert X_{I\cap J} \mathfrak{u}_{I\cap J} + X_{I^{c}\cup J^{c}} \mathfrak{u}_{I^{c}\cup J^{c}} \rVert_{2} \ge \lVert X_{I^{c}\cup J^{c}} \mathfrak{u}_{I^{c}\cup J^{c}} \rVert_{2} - \lVert X_{I\cap J} \mathfrak{u}_{I\cap J} \rVert_{2},
\end{equation}
it follows that
\begin{equation}\label{eq: 1039}
	\lVert X_{I^{c}\cup J^{c}} \mathfrak{u}_{I^{c}\cup J^{c}} \rVert_{2} \le s_{\min}(X) + \lVert X_{I\cap J} \mathfrak{u}_{I\cap J} \rVert_{2}.
\end{equation}

Now let us fix an index set $I\subset [n]$ satisfying $|I|=(1-\eps)n$.
If $\lVert \mathfrak{u}_{I} \rVert_{2} < \delta_{\mathsf{a}}$, we have
\begin{equation*}
	\lVert \mathfrak{u}_{I^{c}\cup J^{c}} \rVert_{2}^{2} = 1 - \lVert \mathfrak{u}_{I\cap J} \rVert_{2}^{2} > 1 - \delta_{\mathsf{a}}^{2}.
\end{equation*}
On the event $\{\lVert \mathfrak{u}_{I} \rVert_{2} < \delta_{\mathsf{a}}\}\cap\mathcal{E}_{2}\cap\mathcal{E}_{3}$, we get
\begin{equation}\label{eq: 1045}
	s_{\min}(X_{I^{c}\cup J^{c}}) \le \frac{1}{\sqrt{1-\delta_{\mathsf{a}}^{2}}} (1-\sqrt{1/\mathsf{a}}+ t + C_{\ref{prop: 716}}\delta_{\mathsf{a}})\sqrt{N}.
\end{equation}
Let $\mathcal{D}_{r}\subset [N]$ be the index set defined by
\begin{equation*}
	\mathcal{D}_{r} = \{i\in[N]: \text{$i$-th row of $X$ has an entry $x_{ij}$ satisfying $|x_{ij}|>N^{\frac{1}{2}-c_{\ref{prop: 709}}}$}\}.
\end{equation*}
Taking the minor $X_{I^{c}\cup J^{c}}^{[\mathcal{D}_{r}]}$ of $X_{I^{c}\cup J^{c}}$ by removing all rows indexed by $\mathcal{D}_{r}$, by the Cauchy interlacing, we observe that
\begin{equation*}
	s_{\min}(X_{I^{c}\cup J^{c}}^{[\mathcal{D}_{r}]}) \le s_{\min}(X_{I^{c}\cup J^{c}}).
\end{equation*}
Then the event $\mathcal{E}_{\textnormal{bad}}\cap\mathcal{E}_{2}\cap\mathcal{E}_{3}$ implies that
there exists a subset $I\subset [n]$ such that $|I|=(1-\eps) n$ and
\begin{equation*}
	s_{\min}(X_{I^{c}\cup J^{c}}^{[\mathcal{D}_{r}]}) \le \frac{1}{\sqrt{1-\delta_{\mathsf{a}}^{2}}} (1-\sqrt{1/\mathsf{a}}+ t + C_{\ref{prop: 716}}\delta_{\mathsf{a}})\sqrt{N}.
\end{equation*}
By the union bound, we obtain
\begin{multline}\label{eq: 1036}
	\prob(\mathcal{E}_{\textnormal{bad}} \cap\mathcal{E}_{1}\cap\mathcal{E}_{2}\cap\mathcal{E}_{3}) \\
	\le {n \choose (1-\eps) n} \cdot \max_{|I|=(1-\eps) n} \prob\bigg(\Big\{s_{\min}(X_{I^{c}\cup J^{c}}^{[\mathcal{D}_{r}]})\le\frac{(1-\sqrt{1/\mathsf{a}}+ t + C_{\ref{prop: 716}}\delta_{\mathsf{a}})\sqrt{N}}{\sqrt{1-\delta_{\mathsf{a}}^{2}}}\Big\}\cap\mathcal{E}_{1}\bigg).
\end{multline}
In addition, on the event $\mathcal{E}_{1}$, we can observe that for a subset $I\subset [n]$ with $|I|=(1-\eps) n$
\begin{equation*}
	N-N^{1-c'_{\ref{prop: 709}}}\le N-|\mathcal{D}_{r}| \le N \quad\text{and}\quad
	\eps n \le |I^{c}\cup J^{c}| \le \eps n + N^{1-c'_{\ref{prop: 709}}}.
\end{equation*}

Recall $\Phi,A$ and $\widetilde{\mathcal{D}}_{c}$ in the proof of Proposition \ref{prop: 716}. We further define the sets $\widetilde{\mathcal{D}}\equiv\widetilde{\mathcal{D}}(\Phi)$ and $\widetilde{\mathcal{D}}_{r}\equiv \widetilde{\mathcal{D}}_{r}(\Phi)$ by 
\begin{align*}
		\widetilde{\mathcal{D}} \equiv \{(i,j)\in [N]\times [n]: \phi_{ij} = 0\}, \quad
			\widetilde{\mathcal{D}}_{r} \equiv  \{ i\in [N]: \sum\nolimits_{j \in [n] } \phi_{ij} < n  \}.
\end{align*}
We denote by $A_{I^{c}\cup \widetilde{\mathcal{D}}_{c}}$ the minor of $A$ which consists of the columns indexed by $I^{c}\cup \widetilde{\mathcal{D}}_{c}$. We obtain $A_{I^{c}\cup \widetilde{\mathcal{D}}_{c}}^{[\widetilde{\mathcal{D}}_{r}]}$ from $A_{I^{c}\cup \widetilde{\mathcal{D}}_{c}}$ by removing all rows indexed by $\widetilde{\mathcal{D}}_{r}$.
Then we notice that
\begin{multline}
	\prob\bigg(\Big\{s_{\min}(X_{I^{c}\cup J^{c}}^{[\mathcal{D}_{r}]})\le\frac{(1-\sqrt{1/\mathsf{a}}+ t + C_{\ref{prop: 716}}\delta_{\mathsf{a}})\sqrt{N}}{\sqrt{1-\delta_{\mathsf{a}}^{2}}}\Big\}\cap\mathcal{E}_{1}\bigg) \\
	= \E\bigg[ \prob_{\Phi}\Big\{s_{\min}(A_{I^{c}\cup \widetilde{\mathcal{D}}_{c}}^{[\widetilde{\mathcal{D}}_{r}]})\le\frac{(1-\sqrt{1/\mathsf{a}}+ t + C_{\ref{prop: 716}}\delta_{\mathsf{a}})\sqrt{N}}{\sqrt{1-\delta_{\mathsf{a}}^{2}}}\Big\}
	\cdot\indic(|\widetilde{\mathcal{D}}|\le N^{1-c'_{\ref{prop: 709}}}) \bigg].
\end{multline}
Now conditioning on $\Phi$ satisfying $|\widetilde{\mathcal{D}}|\le N^{1-c'_{\ref{prop: 709}}}$, we shall bound
\begin{equation*}
	\prob_{\Phi}\Big\{s_{\min}(A_{I^{c}\cup \widetilde{\mathcal{D}}_{c}}^{[\widetilde{\mathcal{D}}_{r}]})\le
	\frac{(1-\sqrt{1/\mathsf{a}}+ t + C_{\ref{prop: 716}}\delta_{\mathsf{a}})\sqrt{N}}{\sqrt{1-\delta_{\mathsf{a}}^{2}}}\Big\}.
\end{equation*}
We choose $I_0$ to be any index set such that $I_0\subset I$ and $|I_0|=n/2$.
Recalling the identity \eqref{eq: 97} for the smallest singular value, it is immediate that
\begin{equation*}
	s_{\min}(A_{I_{0}^{c}\cup \widetilde{\mathcal{D}}_{c}}^{[\widetilde{\mathcal{D}}_{r}]}) \le s_{\min}(A_{I^{c}\cup \widetilde{\mathcal{D}}_{c}}^{[\widetilde{\mathcal{D}}_{r}]}).
\end{equation*}
Let us write $N'=N-|\widetilde{\mathcal{D}}_{r}|$ and $n'=|I_0^{c}\cup \widetilde{\mathcal{D}}_{c}|$. Notice that $n'=n/2(1+o(1))$ according to our choice of $I_0$. 
Set $t=C_{\ref{prop: 716}}\delta_{\mathsf{a}}$.
We can find small constants $\delta_{\mathsf{a}},\gamma_{\mathsf{a}}>0$ depending on $\mathsf{a}$ such that the following holds for large $n$ and $\eps\in(0,1/2)$:
\begin{equation*}
	\gamma_{\mathsf{a}} < \frac{1}{\E[a_{11}^{2}]^{1/2}\cdot\sqrt{N'}} \Big( \E[a_{11}^{2}]^{1/2}\cdot\big(1-\sqrt{n'/N'}\big)\sqrt{N'} - \frac{(1-\sqrt{1/\mathsf{a}}+ t + C_{\ref{prop: 716}}\delta_{\mathsf{a}})\sqrt{N}}{\sqrt{1-\delta_{\mathsf{a}}^{2}}} \Big).
\end{equation*}
This implies
\begin{multline}
    \prob_{\Phi}\Big\{s_{\min}(A_{I^{c}\cup \widetilde{\mathcal{D}}_{c}}^{[\widetilde{\mathcal{D}}_{r}]})\le
	\frac{(1-\sqrt{1/\mathsf{a}}+ t + C_{\ref{prop: 716}}\delta_{\mathsf{a}})\sqrt{N}}{\sqrt{1-\delta_{\mathsf{a}}^{2}}}\Big\} \\
	\le \prob_{\Phi}\Big\{s_{\min}(A_{I_{0}^{c}\cup \widetilde{\mathcal{D}}_{c}}^{[\widetilde{\mathcal{D}}_{r}]})\le
	\E[a_{11}^{2}]^{1/2}\cdot\Big(1-\sqrt{\frac{n'}{N'}}-\gamma_{\mathsf{a}}\Big)\sqrt{N'}\Big\}.
\end{multline}

From now on, we abuse the notation for brevity;
\begin{equation*}
	A=A_{I_{0}^{c}\cup \widetilde{\mathcal{D}}_{c}}^{[\widetilde{\mathcal{D}}_{r}]}.
\end{equation*}

Let $M>0$ be a (large) constant to be chosen later. Define $\widetilde{A}=(\tilde{a}_{ij})$ by $\tilde{a}_{ij}=a_{ij}\indic(|a_{ij}|\le M)$. Also write $\theta=a_{11}\indic(|a_{11}|> M)$.
Applying Theorem \ref{thm: 1301} together with $Y=(y_{ij})=\E[a_{11}^{2}]^{-1/2}A$,
for any $\eta>0$, we can find a constant $c(M,\eta)>0$ depending on $M$ and $\eta$ such that 
\begin{equation*}
	\prob\{ s_{\min}(Y) < s_{\min}(\widetilde{Y}) - \eta\sqrt{N'} - C\sqrt{N'\E\theta^{2}} \}\le \exp(-c(M,\eta)N'),
\end{equation*}
where 
$\widetilde{Y}=(\tilde{y}_{ij})$ is an $N\times n$ matrix with the entries $\tilde{y}_{ij}=y_{ij}\indic\{|y_{ij}|\le M\}-\E[y_{ij}\indic\{|y_{ij}|\le M\}]$, $\theta=y_{11}\indic\{|y_{11}|> M\}+\E[y_{11}\indic\{|y_{11}|\le M\}]$ and $C>0$ is a universal constant.
Consequently,
\begin{multline}
	\prob\Big\{s_{\min}(Y)\le \Big(1-\sqrt{\frac{n'}{N'}}-\gamma_{\mathsf{a}}\Big)\sqrt{N'} \Big\} \\
	\le \prob\Big\{s_{\min}(\widetilde{Y}) - \eta\sqrt{N'} - C\sqrt{N'\E\theta^{2}}\le \Big(1-\sqrt{\frac{n'}{N'}}-\gamma_{\mathsf{a}}\Big)\sqrt{N'} \Big\} + \exp(-c(M,\eta)\cdot N').
\end{multline}
By choosing sufficiently large $M$ and sufficiently small $\eta$, we need to bound the following probability:
\begin{equation*}
	\prob\Big\{s_{\min}(\widetilde{Y}) \le \E[\tilde{y}_{11}^{2}]^{1/2}\cdot\Big(1-\sqrt{\frac{n'}{N'}}-\frac{\gamma_{\mathsf{a}}}{2}\Big)\sqrt{N'} \Big\}.
\end{equation*}
Then, by Corollary V.2.1 in \cite{FS10} (that is applicable due to Assumption \ref{assump: 89}), we have for some constant $\mathfrak{c}>0$
\begin{equation}\label{eq: 1125}
	\prob\Big\{s_{\min}(\widetilde{Y}) \le \E[\tilde{y}_{11}^{2}]^{1/2}\cdot\Big(1-\sqrt{\frac{n'}{N'}}-\frac{\gamma_{\mathsf{a}}}{2}\Big)\sqrt{N'} \Big\} \\
	\lesssim \frac{1}{1-\sqrt{n'/N'}} \exp\Big(-\mathfrak{c}\Big(\frac{\gamma_{\mathsf{a}}}{2}(1-\sqrt{1/\mathsf{a}})\Big)^{3/2}n'\Big),
\end{equation}
where we further assume that $\gamma_{\mathsf{a}}/2<1-\sqrt{1/\mathsf{a}}$ without loss of generality.
Set $\eps\in(0,\eps_{\mathsf{a}})$.
Combining \eqref{eq: 1036}--\eqref{eq: 1125}, the desired result follows by choosing $\eps_{\mathsf{a}}$ sufficiently small.
Note that our choice of $\gamma_{\mathsf{a}}$ does not depend on $\eps$ and also $n'=n/2(1+o(1))$.
\end{proof}

\begin{proof}[Proof of Corollary \ref{cor: 188}]

Define the events $\mathcal{E}_{\textnormal{bad}}^{(k)}$ by setting
\[
\mathcal{E}_{\textnormal{bad}}^{(k)} = \{ \exists I \subset [n] : |I| = (1-\epsilon)n, \lVert \mathfrak{u}^{(k)}_{I} \rVert_2 < \delta_{\mathsf{a}} \}.
\]
Let $\mathcal{E}_{1}$ and $\mathcal{E}_{2}$ be as defined in the proof of Theorem \ref{thm: 133}. We define the event $\mathcal{E}_{3}^{(k)}$ through
\begin{equation*}
	\mathcal{E}_{3}^{(k)}=\{ N^{-1/2}s_{n-k}(X) < 1-\sqrt{1/\mathsf{a}} + t \},
\end{equation*}
where we choose $t > 0$ as in the proof of Theorem \ref{thm: 133}.
Similar to Proposition \ref{prop: 1029}, it follows that for any $K>0$, 
\begin{equation}\label{eq: 1536}
	\prob\{ N^{-1/2}s_{n-k}(X) \ge 1-\sqrt{1/\mathsf{a}} + t \} = O(N^{-K}).
\end{equation} 
Again, by Proposition \ref{prop: 709}, Proposition \ref{prop: 716}, and \eqref{eq: 1536}, for any constant $K > 0$, one finds that
\[
\prob(\mathcal{E}_{1}^{c} \cup \mathcal{E}_{2}^{c} \cup (\mathcal{E}_{3}^{(k)})^{c}) = O(N^{-K}).
\]
Then, following \eqref{eq: 1035} and \eqref{eq: 1039} similarly, we deduce that \eqref{eq: 1045} holds on the event 
\begin{align*}\{\lVert \mathfrak{u}^{(k)}_{I} \rVert_{2} < \delta_{\mathsf{a}}\}\cap\mathcal{E}_{2} \cap \mathcal{E}_{3}^{(k)}.
\end{align*} 
The remaining steps are identical.
\end{proof}

\section{Technical lemmas}\label{sec: technicality}

First we state Seginer's result about expected norms of random matrices, which is helpful when we
control an operator norm of a minor.

\begin{proposition}[Corollary 2.2 in \cite{Seginer00}]\label{prop: Seginer}
Let $Y=(y_{ij})$ be $m_{1}\times m_{2}$ random matrix such that $y_{ij}$s are i.i.d.~zero mean random variables.
Denote by $Y_{i\cdot}$ and $Y_{\cdot j}$ the $i$-th row and $j$-th column of $Y$ respectively.
Then, there exists a constant $C$ such that for any $m_{1},m_{2}\ge 1$ and any $q\le 2\log\max\{m_{1},m_{2}\}$, the following inequality holds:
\begin{equation*}
	\E\lVert Y\rVert^{q}
	\le C^{q}\bigg( \E\max_{1\le i\le m_{1}}\lVert Y_{i\cdot}\rVert^{q}_{2} + \E\max_{1\le j\le m_{2}}\lVert Y_{\cdot j}\rVert^{q}_{2} \bigg).
\end{equation*} 
\end{proposition}

The second statement is due to Rogozin. Combining this with the small ball estimate of Rudelson and Vershynin \cite{RV15}, Tikhomirov derived Lemma \ref{lem: 1081} below. 

\begin{theorem}[Theorem 1 in \cite{Rogozin61}]\label{thm: 555}
Let $\{Z_{j}\}_{j=1}^{k}$ be independent random variables for some positive integer $k$.
Let
\begin{equation*}
	s_{1},s_{2},\cdots,s_{k}>0
\end{equation*}
be some constants.
There exists a universal constant $C_{\ref{thm: 555}}>0$ such that
\begin{equation*}
	\mathcal{Q}\bigg(\sum_{j=1}^{k}Z_{j},s\bigg)\le C_{\ref{thm: 555}} \cdot s \cdot \bigg(\sum_{j=1}^{k}\big(1-\mathcal{Q}(Z_{j},s_{j})\big)s_{j}^{2}\bigg)^{-1/2},
\end{equation*}
for every $s$ satisfying $s \ge s_{j}$ for all $j=1,2,\cdots,k$.
\end{theorem}

\begin{lemma}[Corollary 6 in \cite{Tikhomirov16}] \label{lem: 1081}
Let $\mathbf{x}=(\mathbf{x}_{1},\mathbf{x}_{2},\cdots,\mathbf{x}_{k})$ be a random vector in $\mathbb{R}^{k}$ with independent coordinates such that
\begin{equation*}
	\mathcal{Q}(\mathbf{x}_{i},t)\le 1-\tau,\quad i=1,2,\cdots,k
\end{equation*}
for some constants $t>0$ and $\tau\in(0,1)$. Then, for any $d\in\{1,2,\cdots,k\}$, positive integer $\ell$ and any $d$-dimensional non-random subspace $E\subset\mathbb{R}^{k}$,
\begin{equation*}
	\mathcal{Q}(\textnormal{Proj}_{E}\mathbf{x},t\sqrt{d}/\ell)\le (C_{\ref{lem: 1081}}/\sqrt{\ell\tau})^{d/\ell},
\end{equation*}
where $C_{\ref{lem: 1081}}>0$ is a (sufficiently large) universal constant that may depend on the universal constant $C_{\ref{thm: 555}}$ of Theorem \ref{thm: 555}.
\end{lemma}

Lastly, let us record a technical result which was originally introduced in \cite{Tikhomirov15} to prove the Bai-Yin law for the smallest singular value under an optimal condition.

\begin{theorem}[Theorem 15 in \cite{Tikhomirov15}]\label{thm: 1301}
Let $\xi$ be a random variable with zero mean and unit variance. For any $M>0$ and $\eta>0$, there exist constants $C_{\ref{thm: 1301}}>0$ and $c_{\ref{thm: 1301}}>0$ such that the following holds.
Let $Y=(y_{ij})$ be an $N\times n$ random matrix with i.i.d.~entries distributed as $\xi$ and $n\le N$.
Further, let $\widetilde{Y}=(\tilde{y}_{ij})$ be an $N\times n$ matrix with the entries $\tilde{y}_{ij}=y_{ij}\indic\{|y_{ij}|\le M\}-\E[y_{ij}\indic\{|y_{ij}|\le M\}]$ and denote $\theta=\xi\indic\{|\xi|> M\}+\E[\xi\indic\{|\xi|\le M\}]$. Then, for large $N$, we have 
\begin{equation*}
	\prob\{ s_{\min}(Y) \ge s_{\min}(\widetilde{Y}) - \eta\sqrt{N} - C_{\ref{thm: 1301}}\sqrt{N\E[\theta^{2}]} \} \ge 1 - \exp(-c_{\ref{thm: 1301}}N).
\end{equation*}
\end{theorem}

\end{document}